\documentclass[11pt,fullpage]{article}
\usepackage{amsmath}
\usepackage{amsfonts,amssymb}
\usepackage{amsthm}
\usepackage{hyperref}
\usepackage{graphics,graphicx}
\usepackage{color}
\usepackage[right = 2.5cm, left=2.5cm, top = 2.5cm, bottom =2.5cm]{geometry}
\newcommand{\Real}{\mathbb R}
\newcommand{\norm}[1]{\|#1\|}
\newcommand{\abs}[1]{\left\vert#1\right\vert}
\newcommand{\set}[1]{\left\{#1\right\}}

\newcommand{\grad}{\nabla}

\newcommand{\mollify}[1]{ \mathcal{J}_\epsilon #1 }
\newcommand{\conv}[2]{#1 \ast #2}

\newcommand{\K}{\mathcal{K}}
\newcommand{\N}{\mathcal{N}}

\newcommand{\wknorm}[2]{\norm{#1}_{L^{#2,\infty}}}
\newcommand{\wkspace}[1]{L^{#1,\infty}}
\newcommand{\F}{\mathcal{F}}

\newcommand{\reg}[1]{#1^\epsilon}

\newcommand{\kernel}{\mathcal{K}}

\newcommand{\pd}{D_T}

\newtheorem{theorem}{Theorem}
\theoremstyle{remark}
\newtheorem{remark}{Remark}

\theoremstyle{theorem}

\newtheorem{proposition}{Proposition}
\theoremstyle{definition}
\newtheorem{definition}{Definition}
\theoremstyle{lemma}
\newtheorem{lemma}{Lemma}
\theoremstyle{corollary}

\begin{document}

\title{Inhomogeneous Patlak-Keller-Segel models and Aggregation Equations with Nonlinear Diffusion in $\Real^d$} 

\author{Jacob Bedrossian\footnote{\textit{jacob@cims.nyu.edu}, New York University, Courant Institute of Mathematical Sciences. Partially supported by NSF Postdoctoral Fellowship in Mathematical Sciences, DMS-1103765},  Nancy Rodr\'iguez\footnote{\textit{nrodriguez@math.stanford.edu},Stanford University, Department of Mathematics. Partially supported by NSF Postdoctoral Fellowship in Mathematical Sciences, DMS-1103769}}

\date{}

\maketitle

\begin{abstract}
Aggregation equations and Patlak-Keller-Segel (PKS) models for chemotaxis with nonlinear diffusion are popular models for nonlocal aggregation phenomenon and are a source of a number of interesting mathematical problems in nonlinear PDE.
The purpose of this work is twofold. 
First, we continue our previous work \cite{BRB10}, which focused on nonlocal aggregation, modeled with a convolution. 
 The goal was to unify the local and global theory of these convolution-type models, including the identification of a sharp critical mass; however, some cases involving unbounded domains were left open. In particular, the biologically relevant case $\Real^2$ was not treated. 
In this paper, we present an alternative proof of local existence, which now applies to $\Real^d$ for all $d \geq 2$ and give global results that were left open \cite{BRB10}. 
The proof departs from \cite{BertozziSlepcev10,BRB10} in that it uses a more direct and intuitive regularization that constructs approximate solutions on $\Real^d$ instead of on sequences of bounded domains. 
Second, this work develops the local, subcritical, and small data critical theory for a variety of Patlak-Keller-Segel models with spatially varying diffusion and decay rate of the chemo-attractant.
\end{abstract} 

%%%%%%%%%%%%%%%%%%%%%%%%%%%%%%%%%%%%%%%%%%%%%%%%%%%%%%%%%%%%%%%%%%%%%%%%%%%%%%%%%%%%%%%%%%%%%%%%
\section{Introduction} 
In this paper we study several types of aggregation models with nonlinear diffusion and nonlocal self-attraction. 
The primary focus is to develop and extend the relevant local, subcritical, and small data critical/supercritical theory. 
These results exist in the \emph{perturbative regime} as they all fundamentally treat the PDE as a nonlinear perturbation of the diffusion equation (see below for more information). 
Furthermore, we also present several non-perturbative global existence results for a class of critical problems as well. \\

The first general class of systems we study are those where the nonlocal self-attraction arises as the result of a convolution operator
\begin{equation} \label{def:ADD}
\left\{
\begin{array}{l}
  u_t + \nabla \cdot (u \nabla \K \ast u) = \Delta A(u), \\
  u(0,x) = u_0(x) \in L_+^1(\Real^d;(1+\abs{x}^2)dx)\cap L^\infty(\Real^d),
\end{array}
\right.
\end{equation}
where  $L_+^1(\Real^d;\mu) := \set{f \in L^1(\Real^d;\mu): f \geq 0}$.  
Equation \eqref{def:ADD} is of interest in mathematical biology as it models the competition between a species' desire to aggregate and to disperse.
Dispersal is modeled via the, potentially nonlinear, diffusion term $\Delta A(u)$ and the aggregation is modeled via the nonlocal advective term $\nabla\cdot(u\nabla \K\ast u)$.
The most well-known example of \eqref{def:ADD} is the parabolic-elliptic Patlak-Keller-Segel model, based on the original models of Patlak \cite{Patlak} and Keller and Segel \cite{KS}. 
For more information on the modeling aspects, see the \cite{Hortsmann,HandP} for reviews of various chemotaxis models and \cite{TopazBertozziLewis06,Bio,Milewski,GurtinMcCamy77,Burger07} for more general swarming and aggregation models. 
In this paper, we extend our recent work \cite{BRB10} to provide a more satisfactory and complete local and global theory for \eqref{def:ADD} on $\Real^d$ for $d\geq 2$. 
In \cite{BRB10}, we
studied the local and global existence and uniqueness of bounded, integrable solutions to \eqref{def:ADD} in bounded domains for $d\geq 2$ and all space for $d\geq 3$. 
The primary goal of the previous 
work was to unify the existing Patlak-Keller-Segel global existence theory \cite{SugiyamaDIE06,SugiyamaADE07,SugiyamaDIE07,BlanchetEJDE06,Kowalczyk05,CalvezCarrillo06,Blanchet09} with the local existence and uniqueness theory for less singular versions of \eqref{def:ADD} \cite{BertozziSlepcev10}.
In that work , the $\Real^2$ case was not treated due to several technical difficulties. 
Since this case is very important for mathematical biology, we make specific effort to treat this case and discuss the difficulties in more detail below.  
 \\
    
We present an alternative proof of local existence of \eqref{def:ADD} for a wide range of $A$ and $\K$ which applies on $\Real^d$, $d\geq 2$, for solutions with bounded second moment.
The new proof is based on a regularization directly 
 on $\Real^d$, in contrast to \cite{BRB10}. 
 One of the benefits of this regularization is that it allows one to rigorously justify the application of homogeneous Sobolev embeddings in formal arguments, which are crucial in deducing small data global existence and uniform boundedness in supercritical cases. We prove such results below following iteration techniques similar to those employed in \cite{Corrias04,SugiyamaDIE06,BedrossianIA10}.
We also expand the global existence results of \cite{BRB10} to estimate the critical mass for kernels with a logarithmic 
 singularity at the origin in $\Real^d$.
\\

The second class of nonlocal aggregation models we study is the variable-coefficient parabolic-elliptic Patlak-Keller-Segel model, 
\begin{equation} \label{def:InhKS}
\left\{
\begin{array}{l}
  u_t + \nabla \cdot (u \nabla c) = \Delta A(u), \\
  -\grad \cdot (a(x)\nabla c) + \gamma(x)c = u, \\ 
  u(0,x) = u_0(x) \in L_+^1(\Real^d;(1+\abs{x}^2)dx)\cap L^\infty(\Real^d).
\end{array}
\right.
\end{equation}
We assume $a(x) \in C^1$ is strictly positive and $\gamma(x) \in L^\infty$ is non-negative.  
This PDE system is in many ways similar to \eqref{def:ADD}, however they are not of the same form, as the solution to 
the equation for the chemical concentration, $c$, cannot be written in convolution form unless $a(x)$ and $\gamma(x)$ are both constant. 
In this paper we develop the local, subcritical and small data critical/supercritical theory for \eqref{def:InhKS}. 
The proofs are analogous to those of \eqref{def:ADD} with additional complications arising due to the different nature of the estimates for $c,\grad c, D^2c$ in terms of $u$ (see Appendix). 
Analysis of the critical case and the identification of the critical mass has been completed by I. Kim and one of the authors in a separate work \cite{BedrossianKim10}.\\

One of the most important properties of \eqref{def:ADD} and \eqref{def:InhKS} is that each dissipate the following \emph{free energy} 
\begin{align}
\F(u(t)) & = S(u(t)) - \mathcal{W}(u(t)). \label{def:F}
\end{align}
The \emph{entropy}, $S(u(t))$, and the \emph{interaction energy (potential energy)}, $\mathcal{W}(u(t))$, are given by
\begin{align*}
S(u(t)) & = \int \Phi(u(x,t)) dx, \\ 
\mathcal{W}(u(t)) & = \frac{1}{2}\int u(x,t)c(x,t) dx, 
\end{align*}
with $c(x,t) = \K \ast u$ if the system is of convolution type.  
The \emph{entropy density (internal energy density)}, $\Phi(z)$, is a strictly convex function satisfying
\begin{equation}
\Phi^{\prime\prime}(z) = \frac{A^\prime(z)}{z}, \;\;\; \Phi^\prime(1) = 0, \;\;\; \Phi(0) = 0. \label{def:G}
\end{equation}
In fact, both \eqref{def:ADD} and \eqref{def:InhKS} are formally the gradient flows for \eqref{def:F} in the Euclidean Wasserstein metric (see e.g. \cite{AmbrosioGigliSavare}). 
For \eqref{def:ADD} and \eqref{def:InhKS} there is no fully developed theory for making this precise; however, some aspects can be recovered and have proven very useful \cite{BlanchetCalvezCarrillo08,BlanchetCarlenCarrillo10}. 
In any case, the free energy \eqref{def:F} plays an important role, especially in the global theory as in for example \cite{SugiyamaDIE06,BlanchetEJDE06,CalvezCarrillo06,Blanchet09,Blanchet08,BRB10}.   

\subsubsection*{Notation and Conventions}
We work on $\Real^d$ for $d\geq 2$.  For notational simplicity we denote the parabolic domain by $\Real^d_T = \Real^d\times [0,T]$ and the standard $L^p$ norm by 
$\norm{u}_p=\norm{u}_{L^p(\Real^d)}.$  
We also introduce the following notation for the $k^{th}$ moments:
\begin{align*}
\mathcal{M}_k(f) = \int \abs{x}^k\abs{f(x)}\;dx.
\end{align*}
We let $\eta(x) \in C_c^\infty(\Real^d)$ with $0 < \eta(x) \leq 1$ for $\abs{x} < 1$, $\eta(x) = 0$ for $\abs{x} \geq 1$, and $\eta(x) \equiv 1$ for $\abs{x} \leq 1/2$ be our canonical cut-off function and denote the standard mollifer $\mollify{v} := \eta \ast v$.\\

We use $\mathcal{N}$ to denote the Newtonian potential: 
\begin{equation*}
\mathcal{N}(x) = 
\left\{
\begin{array}{ll}
  \frac{1}{2\pi}\log \abs{x} & d = 2 \\ 
  \frac{\Gamma(d/2 + 1)}{d(d-2)\pi^{d/2}}\abs{x-y}^{2-d} & d \geq 3. 
\end{array}
\right.
\end{equation*}
By `weighted Young's inequality' we mean for $a,b > 0$ and $1 = \frac{1}{p} + \frac{1}{q}$ and $\epsilon > 0$,  
\begin{equation*}
ab \leq \epsilon^p\frac{a^p}{p} + \epsilon^{-q}\frac{b^q}{q}.
\end{equation*} 
Since we will be working with many largely irrelevant constants, we use the notation $f\lesssim_{p,k,...} g$ 
to denote $f\leq C(p,k,..)g$, where $C(p,k,..)$ is a generic constant that depends on $p,k$ etc.	    
%%%%%%%%%%%%%%%%%%%%%%%%%%%%%%%%%%%%%%%%%%%%%%%%%%%%%%%%%%%%%%%%%%%%%%

\subsection{Definitions and Assumptions}
We consider the general class of kernels introduced in \cite{BRB10}, which includes fundamental solutions to elliptic PDEs and other commonly considered attractive kernels.

\begin{definition}[Admissible Kernel] \label{def:admK}
We say a kernel $\K \in C^3 \setminus \set{0}$ is \emph{admissible} if $\K \in W^{1,1}_{loc}$ and the following holds:
\begin{itemize}
\item[\textbf{(KN)}] $\K$ is radially symmetric, $\K(x) = k(\abs{x})$ and $k(\abs{x})$ is non-increasing.
\item[\textbf{(MN)}] $k^{\prime\prime}(r)$ and $k^\prime(r)/r$ are monotone on $r \in (0,\delta)$ for some $\delta > 0$. 
\item[\textbf{(BD)}] $\abs{D^3\K(x)} \lesssim \abs{x}^{-d-1}$. 
\end{itemize}
\end{definition}

This definition ensures that $\K$ is attractive, well-behaved at the origin, and has second derivatives that define bounded distributions on $L^p$ for $1 < p < \infty$. 
The obvious example of an admissible kernel is the Newtonian potential, which is effectively the most singular admissible kernel both at the origin and at infinity (in the sense that it decays the slowest).
We remark that many of our results (\S\ref{subsec:LocExistK} and \S\ref{sec:ContSubEtc}) still hold if we replace condition \textbf{(KN)} with the assumption $\K(x) = \K(-x)$, allowing for non-radially symmetric and general attractive/repulsive type kernels.\\

We limit ourselves to diffusions that do not spread mass faster than linear diffusion; however, using the techniques of \cite{CalvezCarrillo06} one could also treat cases with fast diffusion.
This is more general than the diffusion considered in \cite{BertozziSlepcev10,BRB10}, which were restricted to degenerate diffusion.

\begin{definition}[Admissible Diffusion Functions] \label{def:admDiff}
We say that the function $A(u)$ is an admissible diffusion function if: 
\begin{itemize}
\item[\textbf{(D1)}] $A\in C^1([0,\infty))\;\text{with}\;A'(z)>0\;\text{for}\;z\in (0,\infty)$ and $A(0) = 0$.
\item[\textbf{(D2)}] $A'(z)>c\;\text{for}\;z > z_c$ for some $c,z_c > 0$. 
\item[\textbf{(D3)}] $A^\prime(z) \leq C_A$ when $z < z_A$ for some $z_A, C_A > 0$. 
\end{itemize}
\end{definition}
\noindent
In particular \textbf{(D3)} implies $A(z) = \mathcal{O}(z)$ as $z \rightarrow 0$ and 
$\int_{\set{u < R}} A(u) \lesssim_{C_A,R} M$ for all $R < \infty$. \\

Following \cite{BRB10,BertozziSlepcev10,BertozziBrandman10} we use the following notion of weak solution, which is stronger than traditional distribution solutions. 
In $d \geq 3$, test functions are taken in $\dot{H}^1$, whereas in $d = 2$ minor adjustments must be made, as discussed below and in \cite{BertozziBrandman10}. 
By density arguments, this is basically the same as taking test functions in $C^\infty_c$ and requiring various regularity
assumptions on the solution.  However, but we prefer the current statement of the definition to emphasize the kind of test functions that we are interested in. 
Taking test functions in these spaces is important for the proof of uniqueness, which is based on an $\dot{H}^{-1}$ stability estimate \cite{BertozziSlepcev10,BertozziBrandman10,BRB10,AzzamBedrossian11}.
 
\begin{definition}[Weak Solution in $\Real^d$, $d\geq 2$] \label{def:WSRD}
Let $A$ and $\kernel$ be admissible, and $u_0\in L^1_+(\Real^d;(1+\abs{x}^2)dx)\cap L^\infty(\Real^d)$.  If $d \geq 3$, a measurable function $u:[0,T] \times \Real^d \rightarrow [0,\infty)$ is a weak solution of \eqref{def:ADD} or \eqref{def:InhKS} if $u\in L^\infty((0,T)\times\Real^d)\cap L^\infty(0,T,L_+^1(\Real^d;(1+\abs{x}^2)dx))$, $A(u)\in L^2(0,T,\dot{H}^1(\Real^d))$, $u\grad c \in L^2(0,T;L^2(\Real^2))$, $u_t\in
  L^2(0,T,\dot{H}^{-1}(\Real^d))$, and for all test functions $\phi \in L^\infty(0,T;\dot{H}^{1}(\Real^d))$, 
\begin{equation*}
\int_0^T <u_t,\phi(t)>_{\dot{H}^{-1} \times \dot{H}^1} dt = -\int_0^T \int \left(\grad A(u) - u\grad c\right)\cdot\grad\phi(t) dx dt. 
\end{equation*}
If $d = 2$, a measurable function $u:[0,T] \times \Real^2 \rightarrow [0,\infty)$ is a weak solution of \eqref{def:ADD} or \eqref{def:ADD} if $u\in L^\infty((0,T)\times\Real^2)\cap L^\infty(0,T,L_+^1(\Real^2;(1+\abs{x}^2)dx))$, $\grad A(u)\in L^2(0,T,L^2(\Real^2))$, $u\grad c \in L^2(0,T;L^2(\Real^2))$, $u_t\in L^2(0,T,\mathcal{V}^\star(\Real^2))$, and for all test functions $\phi(t) \in L^\infty(0,T;\mathcal{V})$ we have, 
\begin{equation*}
\int_0^T <u_t,\phi(t)>_{\mathcal{V}^\star \times \mathcal{V}} dt = -\int_0^T \int \left(\grad A(u) - u\grad c \right)\cdot\grad\phi(t) dx dt,
\end{equation*}
where $\mathcal{V} = \set{f \in L^\infty(\Real^2) : \grad f \in L^2(\Real^2)}$.  When solving \eqref{def:ADD} then $c = \K \ast u$ and when solving \eqref{def:InhKS} then $c(t)$ is the strong solution to  
$-\nabla \cdot (a(x)\nabla c(t)) + \gamma(x) c(t) = u(t)$ which vanishes at infinity.
\end{definition}

\begin{remark} \label{rmk:TimeWSRD}
Due to the regularity we are imposing on our solution, we could equivalently require for all $\phi \in \dot{H}^1$ for a.e. $t \in (0,T)$
\begin{equation*}
<u_t,\phi>_{\dot{H}^{-1} \times \dot{H}^1} = -\int \left(\grad A(u) - u\grad c \right)\cdot\grad\phi(t) dx,   
\end{equation*}
with the modification in $\Real^2$. 
\end{remark}

\begin{remark}
The additional complication in $\Real^2$ is due to the fact that the norm $\norm{\grad f}_2$ is not well-behaved in $\Real^2$, since $\norm{\grad f}_2$ scales like $L^\infty$ in $\Real^2$. Indeed, there exists a sequence of Schwartz functions with $\norm{\grad f_n}_2 = 1$ and $f_n \rightarrow \infty$ point-wise a.e. (consider $f(x) = \log\log(1 + \abs{x}^{-1})\mathbf{1}_{\abs{x} < 1}$ and scaling $f_\lambda(x) = f(\lambda x)$, $\lambda \rightarrow 0$). 
In order to remove such pathologies from our space of test functions we follow \cite{BertozziBrandman10}, which also requires test functions to lie in $L^\infty(\Real^2)$. 
\end{remark}
\vspace{6pt}
Following \cite{BRB10}, we now define a notion of criticality for \eqref{def:ADD}, which in general has no scaling symmetries. 
However, a kind of scaling symmetry can be recovered in the limit of mass concentration, which in turn is expected to govern blow-up (see Theorem \ref{thm:Continuation} below). 
Suppose $\K(x) \sim \abs{x}^{-d/p}$ as $x \rightarrow 0$, $A(u) = u^m$ with $m > 1$ and let $u \in C^\infty_c \cap L_+^1$.  
Then if $u_\lambda(x) = \lambda^{-d}u(\lambda^{-1}x)$ we have, 
\begin{equation*}
\lim_{\lambda \rightarrow 0}\F(u_\lambda) = \lim_{\lambda \rightarrow 0}\left[\frac{\lambda^{d-dm}}{m-1}\int u^m dx - \frac{\lambda^{-d/p}}{2}\int\int u(x)u(y)\abs{x - y}^{-d/p} dxdy\right].
\end{equation*}
From here we see the limit is independent of $u$ unless $d-dm = -d/p$. If the limit is $-\infty$ we expect aggregation to dominate near mass concentration and if the limit is $+\infty$ we expect diffusion to dominate. 
As mass concentration should occur on vanishing length-scales, we may use this scaling heuristic to define a notion of criticality.
The limiting case of $p \rightarrow \infty$ occurs when considering $\K(x) \sim -\log\abs{x}$ as $x \rightarrow 0$ and is discussed more below. Of course this corresponds to the 2D classical parabolic-elliptic PKS model, but we can consider the same singularity in higher dimensions (as done in for example \cite{KarchSuzuki10}).  

\begin{definition}[Critical Exponent]\label{def:mstar}
Suppose $\K$ is admissible such that for some $d/(d-2) \leq p \leq \infty$ (with the convention that $p = \infty$ if $d = 2$) we have $D^2\K(x) = \mathcal{O}\left(\abs{x}^{-d/p - 2}\right)$ as $x \rightarrow 0$. Then the \emph{critical exponent} associated to $\K$ is given by
 
\begin{equation*}
1 \leq m^\star = \frac{p+1}{p} \leq 2 - 2/d.
\end{equation*} 
For the variable-coefficient Patlak-Keller-Segel system \eqref{def:InhKS} we take $m^\star = 2-2/d$. 
\end{definition}
\begin{remark}
Due to the monotonicity assumptions in Definition \ref{def:admK} (see also Lemma \ref{lem:PropAdmissible} above), for $p < \infty$ the definition is equivalent to requiring that $\K(x) = \mathcal{O}(\abs{x}^{-d/p})$ as $x \rightarrow 0$ which is the same as requiring $\K \in L^{p,\infty}_{loc}$. Similarly, when $m^\star = 1$ we have that $\K(x) = \mathcal{O}(\log\abs{x})$ as $x \rightarrow 0$. 
\end{remark}
\vspace{6pt}
\begin{remark}
The variable coefficient system \eqref{def:InhKS} should be roughly as singular as the constant coefficient case, hence the corresponding definition in this case. 
\end{remark}
\vspace{6pt}

Now we define the notion of criticality by relating the critical exponent of the kernel to the diffusion, again focusing on the limit of mass concentration. It is easier to define this notion in terms of the quantity
$A^\prime(z)$, as opposed to using $\Phi(z)$ directly.  This is not so surprising as $A'(z)$ is precisely the local coefficient of diffusivity, and directly measures the strength of the diffusion relative to the mass density. 
\begin{definition}[Criticality] \label{def:criticality}
We say that the problem is \emph{subcritical} if
\begin{equation*}
\liminf_{z \rightarrow \infty} \frac{A^\prime(z)}{z^{m^{\star}-1}} = \infty,
\end{equation*}
\emph{critical} if
\begin{equation*}
0 < \liminf_{z \rightarrow \infty} \frac{A^\prime(z)}{z^{m^{\star}-1}} < \infty,
\end{equation*}
and \emph{supercritical} if
\begin{equation*}
\liminf_{z \rightarrow \infty} \frac{A^\prime(z)}{z^{m^{\star}-1}} = 0.
\end{equation*}
\end{definition}

The following lemma, from \cite{BRB10}, enumerates several important aspects of admissible kernels. 
Part \textit{(c)} in particular provides a useful characterization of kernels with $m^\star < 2-2/d$.
 
\begin{lemma} \label{lem:PropAdmissible}
Let $\K$ be admissible. Then each of the following is true: 
\begin{itemize}
\item[(a)] $\grad \K \in \wkspace{\frac{d}{d-1}}$, and if $d \geq 3$ then $\K \in \wkspace{\frac{d}{d-2}}$. If $d = 2$, then $\K \in BMO(\Real^2)$.
\item[(b)] For all $p$, $1 < p < \infty$, $\exists \;C(p) > 0$ such that for all $f \in L^p(\Real^d)$, $\norm{D^2\K \ast f}_p \leq C(p)\norm{f}_p$. Moreover, 
$C(p) \lesssim p$ as $p \rightarrow \infty$. 
\item[(c)] Let $1\leq m^*<2-2/d$ and $\beta$ be such that $1 < \beta \leq d/2$. Then, $D^2\K \in \wkspace{\beta}_{loc}$ if and only if $\grad \K \in \wkspace{d/(d/\beta - 1)}_{loc}$. 
If $d\geq 3$, then $\K \in \wkspace{d/(d/\beta - 2)}_{loc}$ if and only if $D^2\K \in \wkspace{\beta}_{loc}$ and $m^\star = 1 + 1/\beta - 2/d$ for some 
$1 < \beta \leq d/2$ if and only if $D^2\K\in \wkspace{\beta}_{loc}$. Note by \textbf{(BD)}, if $D^2\K \in \wkspace{\beta}_{loc}$ for some $\beta > 1$, then, $D^2\K \in \wkspace{\beta}$.  
\end{itemize} 
\end{lemma}

We will also need the logarithmic Hardy-Littlewood-Sobolev inequality in order to relate the interaction energy to the Boltzmann entropy, as in for instance \cite{Dolbeault04,BlanchetEJDE06,BRB10}.  

\begin{lemma}[Logarithmic Hardy-Littlewood-Sobolev inequality \cite{CarlenLoss92}] \label{lem:log_sob}
Let $d \geq 2$ and $f \in L_+^1(\Real^d)$ be such that $f \log f \in L^1(\Real^d)$. Then, 
\begin{equation}
-\int\int_{\Real^d\times \Real^d}f(x)f(y) \log\abs{x-y} dxdy \leq \frac{\norm{f}_1}{d}\int_{\Real^d} f\log f dx + C(\norm{f}_1). \label{ineq:log_sob}
\end{equation}
\end{lemma} 
\begin{remark}
One also has for all $\delta > 0$, 
\begin{equation}
-\int\int_{\abs{x-y} < \delta}f(x)f(y) \log\abs{x-y} dxdy \leq \frac{\norm{f}_1}{d}\int_{\Real^d} f\log f dx + C(\norm{f}_1). \label{ineq:log_sob_localized}
\end{equation} 
\end{remark}
Also recall the well-known convolution inequality (see \cite{LiebLoss}): 
for all $f \in L^p(\Real^d)$, $g \in L^q(\Real^d)$ and $\K \in \wkspace{t}$ for $1 < p,q,t < \infty$ satisfying $1/p + 1/q + 1/t = 2$,
\begin{equation}
\abs{\int\int f(x)g(y)\K(x-y)dxdy} \lesssim_{p,q,t,d} \norm{f}_p\norm{g}_q \wknorm{\K}{t}. \label{ineq:GGHLS} 
\end{equation} 

As noted above for $d=2$ the admissible kernels are generally only in $BMO(\Real^2)$ and hence can grow logarithmically at infinity. 
This introduces a number of complications for the local and global well-posedness.
To begin with, in the proof of the energy dissipation inequality, one must ensure that the interaction energy of the approximate solutions converges to the interaction energy of the weak solution being constructed. However, $\K \ast f$ will be unbounded for general $f \in L^1 \cap L^\infty$ hence more care must be taken than in $d \geq 3$.  
The dual of $BMO$ is the Hardy space $\mathcal{H}^1$ \cite{BigStein}, a strict subset of $L^1$, which we define via duality,
\begin{equation}
\norm{f}_{\mathcal{H}^1} := \sup_{\K \in BMO, \norm{\K}_{BMO} = 1} \int \K f dx. \label{def:H1}
\end{equation} 
Accordingly, we have the natural analogue of H\"older's inequality \cite{BigStein}
\begin{equation}
\abs{\int \K f dx} \leq \norm{\K}_{BMO}\norm{f}_{\mathcal{H}^1}, \label{ineq:BMOHo}
\end{equation}
which in particular implies $\K \ast f \in L^\infty(\Real^d)$ whenever $\K \in BMO$ and $f \in \mathcal{H}^1$. 
The following lemma found in \cite{AzzamBedrossian11} provides sufficient conditions for $f\in L^1$ such that $f \in \mathcal{H}^1$ and a useful estimate of the 
norm that supplies the convergence of the interaction energies. 

\begin{lemma}[\cite{AzzamBedrossian11}] \label{lem:Hardy}
Let $f \in L^1 \cap L^p$ for some $p > 1$ and satisfy $\int f dx = 0$, $\mathcal{M}_1(f) < \infty$. 
Then $f \in \mathcal{H}^1$ and 
\begin{equation*}
\norm{f}_{\mathcal{H}^1} \lesssim_{d,p} \norm{f}_{p} + \mathcal{M}_1(f). 
\end{equation*}
\end{lemma}
%%%%%%%%%%%%%%%%%%%%%%%%%%%%%%%%%%%%%%%%%%%%%%%%%%%%%%%%%%%%%%%%%%%%%%%%%%%
\subsection{Statement of Results}
\begin{theorem} [Local Existence and Energy Dissipation for Convolution-type Systems] \label{thm:loc_theory}
Let $d \geq 2$, $\K$ be admissble and $u_0 \in L^1_+(\Real^d;(1 + \abs{x}^2)dx) \cap L^\infty(\Real^d)$. Then there exists a $T > 0$ and 
a weak solution $u(t)$ of \eqref{def:ADD} which satisfies $u(t) \in C([0,T];L_+^1(\Real^d;(1+\abs{x}^2)dx))\cap L^\infty((0,T) \times \Real^d))$ and $u(0) = u_0$.
Moreover, $\F(u_0) < \infty$ and $u(t)$ satisfies the energy dissipation inequality, 
\begin{align}\label{EnrDiss}
\F(u(t))+\int_0^t\int \frac{1}{u(s)}\left|A'(u(s))\nabla u(s) -u(s)\nabla\kernel\ast u(s)\right|^2dxds \leq \F(u_0).  
\end{align}
\end{theorem} 
For a rigorous interpretation of the free energy dissipation, 
\begin{equation*}
\int D[u] dx := \int \frac{1}{u}\abs{A^\prime(u)\grad u - u\grad \K \ast u}^2 dx, 
\end{equation*}
see for example \cite{CarrilloEntDiss01}.
As mentioned above, the local existence for general models \eqref{def:ADD} for $d \geq 3$ was proved in \cite{BRB10}, which 
extended the existence results of \cite{SugiyamaDIE06,BlanchetEJDE06,BertozziSlepcev10}.
We present an alternative which is specialized to treating $\Real^d$ that has certain advantages. In particular we treat $\Real^2$.  
We also prove the corresponding theorem for variable-coefficient Patlak-Keller-Segel systems. 

\begin{theorem} [Local Existence and Energy Dissipation for Variable-Coefficient Systems] \label{thm:loc_theory_variable}
Let $d \geq 2$, $a(x) \in C^1$ be strictly positive such that $a + \abs{\grad a} \in L^\infty$ and let $\gamma(x) \in L^\infty$ be non-negative. In $d = 2$, further suppose that $\gamma(x)$ is strictly positive. 
If $u_0 \in L^1_+(\Real^d;(1 + \abs{x}^2)dx) \cap L^\infty(\Real^d)$ then
there exists a $T > 0$ and 
a weak solution $u(t)$ of \eqref{def:InhKS} which satisfies $u(t) \in C([0,T];L_+^1(\Real^d;(1+\abs{x}^2)dx))\cap L^\infty((0,T) \times \Real^d))$ and $u(0) = u_0$.
Moreover, $\F(u_0) < \infty$ and $u(t)$ satisfies the energy dissipation inequality
\begin{align}\label{EnrDiss2}
\F(u(t))+\int_0^t\int \frac{1}{u(s)}\left|A'(u(s))\nabla u(s) -u(s)\nabla c(s)\right|^2dxdt \leq \F(u_0).  
\end{align}
\end{theorem}   
Uniqueness for convolution-type systems is proven in \cite{BRB10} for $d \geq 3$; the same proof works for \eqref{def:InhKS} using the elliptic estimates found in the Appendix.  
More recent work undertaken by J. Azzam and one of the authors \cite{AzzamBedrossian11} proves uniqueness for the $d = 2$ case.  
For completeness, we state a continuation theorem proved in \cite{BRB10}, which extends previous theorems stated in \cite{JagerLuckhaus92,CalvezCarrillo06,Blanchet09}. The extension to cover \eqref{def:InhKS} is straightforward and is briefly discussed below in \S\ref{sec:ContInhKS}. 
 
\begin{theorem}[Continuation \cite{BRB10}] \label{thm:Continuation}
The weak solution to \eqref{def:ADD} or \eqref{def:InhKS} has a maximal time interval of existence $T_\star = T_\star(u_0)$ and if 
\begin{equation}\label{cond:equint}
\lim_{k \rightarrow \infty}\sup_{t \in [0,T_\star)} \norm{(u-k)_+}_{\frac{2-m}{2-m^\star}} = 0,  
\end{equation}
then $T_\star = \infty$ and $u(t) \in L^\infty( (0,\infty) \times \Real^d)$. 
Here $m$ is such that $1 \leq m \leq m^\star$ and $\liminf_{z \rightarrow \infty}A^{\prime}(z)z^{1-m} > 0$. 
In particular, if $T_\star < \infty$, then for all $p > (2-m)/(2-m^\star)$, 
\begin{equation*}
\lim_{t \nearrow T_\star} \norm{u}_{p} = \infty.
\end{equation*}
\end{theorem}

One of the primary tools in the proofs of Theorems \ref{thm:loc_theory}-\ref{thm:Continuation} is the use of Alikakos iteration methods commonly used in the study of these PDEs as for example \cite{JagerLuckhaus92,Kowalczyk05,CalvezCarrillo06,BRB10,BedrossianIA10}. 
These methods are fundamentally perturbative in nature (as are the related methods of \cite{PerthameVasseur10,BlanchetDEF10}), depending on relatively crude Gagliardo-Nirenberg inequalities to overpower the nonlinear aggregation with diffusion only in certain regimes. 
In subcritical regimes this is sufficient and these methods prove global existence and uniform boundedness in $L^\infty$, as in \cite{Kowalczyk05,CalvezCarrillo06}. 
In the critical and supercritical cases, one can prove the same provided that the initial condition is small in the corresponding critical norm and that the nonlinear diffusion compares favorably with the homogeneous diffusion even at low densities.
As shown in for example \cite{Corrias04,SugiyamaDIE06,PerthameVasseur10,BedrossianIA10}, stronger decay estimates may also be deduced using various refinements of similar iteration methods. 

\begin{theorem}[Subcritical and Small Data Theory] \label{thm:GBDs}
Let $u_0 \in L_+^1(\Real^d; (1+\abs{x}^2)dx) \cap L^\infty(\Real^d)$ and let $u(t)$ be the local-in-time weak solution to \eqref{def:ADD} or \eqref{def:InhKS} with $u(0) = u_0$.
\begin{itemize}
\item[(i)] (subcritical)
Suppose that for $z$ sufficiently large there exists some $\delta > 0$ such that 
\begin{equation}
A^\prime(z) \geq \delta z^{m-1},  \label{cond:Aprime_GBDs2}
\end{equation}
for $m > m^\star$. Then the solution is global and $u(t) \in L^\infty( (0,\infty) \times \Real^d)$.

\item[(ii)] (small data in critical and supercritical cases)  
Suppose instead \eqref{cond:Aprime_GBDs2} is satisfied for all $z > 0$ for $1 \leq m \leq m^\star$. 
If $u(t)$ solves \eqref{def:InhKS} and $m < m^\star$, then assume $\gamma(x)$ is strictly positive. Then 
there exists a constant $\epsilon_0 = \epsilon_0(\delta,m,d) > 0$ such that if 
\begin{equation*}
\norm{u_0}_{\frac{2-m}{2-m^\star}} < \epsilon_0, 
\end{equation*} 
then the solution is global ($T_\star(u_0) = \infty$) and  $u(t) \in L^\infty( (0,\infty) \times \Real^d)$. 
Furthermore, if $\gamma(x)$ is not strictly positive and $d\geq 3$ then there exists a constant $\epsilon_0 = \epsilon_0(\delta,m,d) > 0$ such that if 
\begin{equation*}
\max\left(\norm{u_0}_1,\norm{u_0}_{\frac{2-m}{2-m^\star}}\right) < \epsilon_0, 
\end{equation*} 
then the solution is global and  $u(t) \in L^\infty( (0,\infty) \times \Real^d)$.
\end{itemize}
\end{theorem}

\begin{remark}
Part \textit{(i)} follows easily from the techniques used to prove the Theorem \ref{thm:Continuation} found in \cite{Kowalczyk05,CalvezCarrillo06,BRB10} and will not be proved here. However, \textit{(ii)} is not so immediate, especially for the variable-coefficient system \eqref{def:InhKS} when $\gamma(x)$ is not strictly positive. 
\end{remark}
\vspace{3pt}
\begin{remark}
One can check that if $\K = \N$ and $A(u) = u^{m}$ then if $1 \leq m \leq m^\star$, \eqref{def:ADD} has a scaling symmetry which leaves the norm $L^{(2-m)/(2-m^\star)}$ invariant. It is in this sense that the $L^{(2-m)/(2-m^\star)}$ norm is critical. Hence, even when $m < m^\star$, part \textit{(ii)} is still a small data \emph{critical} result, which is expected. However the proof is more difficult in this case, as \eqref{def:ADD} has no controlled quantity which controls the $L^{(2-m)/(2-m^\star)}$ a priori.  
\end{remark}
\vspace{3pt}
\begin{remark}
We are unsure if the additional requirement on the mass in \textit{(ii)} when $\gamma(x)$ is not strictly positive in supercritical problems is necessary. 
For the Alikakos iteration methods to work, the singularity of the advective term is quantified in terms of the elliptic estimates of Lemmas \ref{lem:HomogGrad} and \ref{lem:inhom_bds_CZ} in the Appendix. In order to prove part \textit{(ii)} for supercritical cases $m < m^\star$, one must use these estimates in a bootstrap argument to also control the critical norm $L^{\frac{2-m}{2-m^\star}}$; however, Lemma \ref{lem:HomogGrad} introduces a `subcritical' lower-order term that seems to require the additional assumption in order to control. 
\end{remark}
\vspace{3pt}
For critical problems, deducing global bounds and decay estimates for larger data requires fully non-pertubative techniques that depend heavily on the energy dissipation inequalities \eqref{EnrDiss} and \eqref{EnrDiss2}. 
The use of sharp functional inequalities to determine when mixed-sign energies such as \eqref{def:F} are coercive is the classical standard method of treating large data and determining the sharp critical mass, for example for the focusing nonlinear Schr\"odinger equations, marginal unstable thin film equations and, of course, PKS \cite{Weinstein83,Witelski04,KillipVisanClay,BlanchetEJDE06,CalvezCarrillo06,Blanchet09,BRB10}.
Our previous work of \cite{BRB10} treated the case $m^\star > 1$ for critical convolution-type systems in $\Real^d$, and in this work 
we complete the $m^\star = 1$ case, for example, now covering the variants of the classical Patlak-Keller-Segel system in $\Real^2$.
For Patlak-Keller-Segel, the proof that the critical mass is sharp follows easily from a standard Virial argument (see e.g. \cite{BlanchetEJDE06,SugiyamaDIE06,SugiyamaDIE07,Blanchet09,BRB10} or the more classical \cite{Nagai95}), which can be modified in a straightforward manner to treat more general problems \cite{BRB10}.    
 The corresponding program for the variable-coefficient systems \eqref{def:InhKS} is a more difficult problem, completed by I. Kim and one of the authors elsewhere \cite{BedrossianKim10}. 
We also stress that \textit{(ii)} is already known, e.g. \cite{BlanchetEJDE06,Blanchet09}, but we restate and prove it to provide comparison with the other problems that do not have scaling symmetries. 

\begin{theorem}[Global Existence in $\Real^d$, Critical case, $m^\star = 1$] \label{thm:GWP} 
Let $A,\K$ be admissible and suppose $u_0 \in L_+^1(\Real^d; (1+\abs{x}^2)dx) \cap L^\infty(\Real^d)$ and let $u(t)$ be the local-in-time weak solution to 
\eqref{def:ADD} with $u(0) = u_0$ and define $M = \norm{u_0}_1$. 
\begin{itemize}
\item[(i)] Suppose the problem is critical. Then there exists a critical mass $M_c$ such that if $M < M_c$ then $u(t)$ exists globally. The estimate of $M_c$ is given below by \eqref{def:MC}. If $\K$ is bounded from below and 
\begin{equation*}
\int_0^1 A^\prime(z)z^{-1}dz  < \infty, 
\end{equation*}
then we may additionally assert $u(t) \in L^\infty( (0,\infty) \times \Real^d)$.   
\item[(ii)] Suppose $m^\star = 2-2/d$, $A(u) = u^{2-2/d}$ and $\mathcal{\K} = \mathcal{N}$, the Newtonian potential.  
Furthermore, suppose that $M < M_c$. Then $u(t)$ exists globally and satisfies
\begin{equation*}
\norm{u(t)}_\infty \lesssim (1 + t)^{-1}.
\end{equation*} 
\end{itemize}
\end{theorem}
\begin{proposition} \label{prop:critical_mass}
Let $\K$ satisify $\K(x) = -c\log \abs{x} + o(\log\abs{x}), \; x \rightarrow 0$ for some $c \geq 0$. 
Then, the critical mass $M_c$ satisifies
\begin{equation}
\lim_{z \rightarrow \infty} \frac{\Phi(z)}{z \log z} - \frac{c}{2d}M_c = 0. \label{def:MC}
\end{equation}
\end{proposition}

\begin{remark}
Note the additional requirements in part \textit{(i)} of Theorem \ref{thm:GWP} in order to assert global boundedness.  
Although we believe that solutions are bounded (and in many cases likely decay as in \textit{(ii)}) our proof cannot rule out an unbounded increase of entropy or interaction energy as the solution spreads without the additional assumptions. 
\end{remark}

%%%%%%%%%%%%%%%%%%%%%%%%%%%%%%%%%%%%%%%%%%%%%%%%%%%%%%%%%%%%%%%%%%%%%%%%%%%%%%%%%%%%%%%%%%%%%%%%%
\section{Local Existence and Energy Dissipation}
%%%%%%%%%%%%%%%%%%%%%%%%%%%%%%%%%%%%%%%%%%%%%%%%%%%%%%%%%%%%%%%%%%%%%
\subsection{Local Existence for Convolution-type Systems} \label{subsec:LocExistK}
This section focuses on the proof of local existence of weak solutions with bounded second moment.  The proofs included here are simplifications of our work 
in \cite{BRB10} and combines techniques from the PKS model found in \cite{BlanchetEJDE06} and the non-singular aggregation-diffusion equations found in 
\cite{BertozziSlepcev10}.  One advantage of the new techniques used here is the treatment of $\Real^2$. 
Furthermore, in contrast to the work in \cite{BRB10}, we obtain solutions directly in $\Real^d$, removing the need for the intermediate 
step of finding solutions in bounded domains.  Consider the regularized aggregation-diffusion equation
\begin{align}\label{def:RAE}
\reg{u}_t =  \Delta \reg{A}(\reg{u})- \nabla \cdot \left(\reg{u}\nabla \reg{\K} \ast \reg{u}\right).
\end{align}
\noindent We define,
\begin{align}\label{ParabolicReg}
\reg{A}(z) = \int^z_0 a_\epsilon'(z) dz,
\end{align}
\noindent where $a_\epsilon'(z)$ is a smooth function, such that $A'(z)+\epsilon \leq a_\epsilon'(z) \leq A'(z)+2\epsilon$. 
We denote,  
\begin{align*}
\reg{\K}(x) := \left\{ 
\begin{array}{ll}
\int \frac{1}{\epsilon^d\eta(x)^d}\eta\left(\frac{x-y}{\epsilon\eta(x)}\right)\K(y) dy & \abs{x} < 1 \\
\K(x) & \abs{x} \geq 1
\end{array}
\right.
\end{align*}
Hence by Definition \ref{def:admK} $\K^\epsilon \in C^3(\Real^d)$.

\begin{proposition}[Local Existence for the Regularized System]\label{prop:LERADD}
Let $\epsilon>0$ be fixed and $u_0(x)\in L_+^1(\Real^d;(1+\abs{x}^2)dx) \cap L^\infty(\Real^d)$. Then \eqref{def:RAE} has a classical solution 
$\reg{u}$ on $[0,T] \times \Real^d$ for all $T> 0$ with $\reg{u}(0) = \mollify u_0$. 
\end{proposition}

To prove the above proposition we need some preliminary definitions.  
Define the Hilbert space
\begin{equation*} 
V = \set{v \in H^1: \norm{v}_V <\infty}, \;\; \norm{v}_V = \sqrt{<v,v>_V},
\end{equation*}  
with the inner product defined via $<u,v>_{V} := <u,v>_{H^1} + \int \abs{x} uv dx$.
Note by the Rellich-Khondrashov compactness theorem, $V$ is compactly embedded in $L^2(\Real^d)$ \cite{Evans}. 
We will construct a weak solution to \eqref{def:RAE} with an analogous definition of weak solutions to Definition \ref{def:WSRD}. 
For the remainder of the paper we denote the mass of the initial data, $u_0$, by $M$, \textit{i.e.} $\norm{u_0}_1 = M$. \\

We prove Proposition \ref{prop:LERADD} using the Schauder fixed point theorem, see e.g. \cite{BenyaminiLindenstrauss}. 
The necessary compactness for the application is obtained via the Aubin-Lions Lemma \cite{Showalter}. 
Now we state and prove some a priori estimates that will be of used in the proof of Proposition \ref{prop:LERADD}.  Some of these estimates are 
the same or closely related to estimates proved elsewhere (e.g. \cite{BlanchetEJDE06,BertozziSlepcev10,BRB10}).  
\begin{lemma}[A priori bounds with linear advection]\label{lem:AprioriBoundLP}
For fixed $\epsilon>0$ let $\tilde{u}\in L^2(0,T;L^2) \cap L^\infty(0,T;L^1)$ be given.  Let $\reg{u}$ be the global strong solution to
\begin{align}\label{eq:LinearProblem}
\reg{u}_t =  \Delta \reg{A}(\reg{u})- \nabla \cdot \left(\reg{u}\nabla \reg{\K} \ast \tilde{u}\right)
\end{align}
\noindent with initial data $\reg{u}_0 = \mollify u_0(x)$ with $u_0 \in L_+^1(\Real^d;(1+\abs{x}^2)dx) \cap L^\infty(\Real^d)$.  Then,
\begin{itemize}
\item[(i)] $\norm{\reg{u}}_{L^2(0,T;L^2)} \leq T^{1/2}\norm{\reg{u}}_{L^\infty(0,T;L^2)}\leq T^{1/2}\norm{u_0}_2\exp\left(\frac{1}{2}\norm{\Delta\reg{\K}}_{2}\norm{\tilde{u}}_{L^2(0,T,L^2)}T^{1/2}\right)$ 
\item[(ii)] $\norm{\reg{u}(t)}_{\infty} \leq \norm{u_0}_\infty\exp\left(\norm{\Delta \reg{\K}}_2\norm{\tilde{u}}_{L^2(0,T,L^2)}T^{1/2}\right)$. 
%\item[(iii)] $\norm{\reg{A}(\reg{u})}_1 \lesssim (1 + \reg{A}(\norm{\reg{u}}_\infty))\norm{\reg{u}}_1$. 
\item[(iii)] $\norm{\grad \reg{A}(\reg{u})}_{L^2(0,T;L^2)}^2 \lesssim \reg{A}(\norm{\reg{u}}_\infty)\norm{\reg{u}}_1 +\norm{\nabla \reg{\K}}^2_{L^\infty}\norm{\tilde{u}}^2_{L^\infty(0,T;L^1)} \norm{\reg{u}}_{L^\infty(0,T;L^2)}^2$.
\item[(iv)]$\norm{\nabla \reg{u}}_{L^2(0,T,L^2(\Real^d))}\lesssim \epsilon^{-1}\norm{\grad \reg{A}(\reg{u})}_{L^2(0,T;L^2)}$. 
\item[(v)] $\mathcal{M}_2(\reg{u}(t)) \leq C(\mathcal{M}_2(u_0),\epsilon,\norm{u}_{L^\infty(0,T;L^\infty)}, \norm{\tilde{u}}_{L^2(0,T;L^2)},\norm{u_0}_1,T)$. 
\item[(vi)] $\norm{\partial_t\reg{u}}_{L^2(0,T, H^{-1})} \leq C(\epsilon,\norm{u}_{L^\infty(0,T;L^\infty)}, \norm{\tilde{u}}_{L^2(0,T;L^2)},\norm{u_0}_1,T)$. 
\end{itemize}
\end{lemma}

\begin{proof}
In what follows denote $M := \norm{\reg{u}(t)}_1 = \norm{u_0}_1$. 
By \eqref{eq:LinearProblem} and integration by parts, once on the diffusion term and twice on the aggregation term we obtain 

\begin{align*}
\frac{d}{dt}\int (\reg{u})^2 dx & \leq -2\int \abs{\reg{A}(\reg{u})}^2 + \int \nabla (\reg{u})^2 \cdot\nabla \reg{\K} \ast \tilde{u} dx \\
& \leq - \int \nabla (\reg{u})^2 \Delta \reg{\K} \ast \tilde{u} dx\\
&\leq  \norm{\Delta \reg{\K}}_{2}\norm{\tilde{u}}_{2}\int (\reg{u})^2 dx. 
\end{align*}  
Integrating implies, 
\begin{equation*}
\norm{u(t)}_2 \leq \norm{u_0}_2\exp\left(\frac{1}{2}\norm{\Delta \reg{\K}}_2 \norm{\tilde{u}}_{L^1(0,T,L^2)}\right) \leq \norm{u_0}_2\exp\left(\frac{1}{2}\norm{\Delta \reg{\K}}_2 \norm{\tilde{u}}_{L^2(0,T,L^2)}T^{1/2}\right). 
\end{equation*}
which gives \textit{(i)}.  
The bound \textit{(ii)} follows similarly by estimating the growth of $\norm{u(t)}_p$ and passing to the limit $p \rightarrow \infty$.
To continue, we need a bound on the $L^1$ norm of $\reg{A}(\reg{u})$.  Condition \textbf{(D3)} of Definition \ref{def:admDiff} implies that $\reg A(z) \leq \left(C_A + 2\epsilon\right)z$ for some $C_A > 0$ and sufficiently small $z$. 
Hence by Chebyshev's inequality,
\begin{align*}
\int \reg{A} (u) dx &= \int_{\set{u<1}} \reg{A}(u) dx + \int_{\set{u\geq 1}} \reg{A} (u) dx\\
&\leq (C_A+2\epsilon)M+ \reg{A}(\norm{u}_\infty) \lambda_u(1) \leq (C_A + 2\epsilon + \reg{A}(\norm{u}_\infty))M.
\end{align*}

We now turn to the less trivial estimate \textit{(iii)}. Let $\eta_R(x) := \eta(xR^{-1})$ for some $R > 0$, where $\eta$ is the smooth cut-off function defined above. Now take $\tilde{A}=\reg{A}(\reg{u})\eta_R$ as a test function in the definition of weak solution (Definition \ref{def:WSRD}), which implies,
\begin{align*}
\int_0^T\left\langle u_t(t), \tilde{A}\right\rangle\;ds &= \int_0^T\int \left(-\nabla \reg{A}(u) + u\nabla \reg{\K} \ast \tilde{u}\right)\cdot \nabla \tilde{A}\; dxds\\
& \leq -\int_0^T\int \grad \reg{A}(\reg{u}) \cdot \grad \tilde{A}(\reg{u}) dx + \frac{1}{2}\int_0^T\hspace{-6pt}\int\hspace{-1pt} \abs{\nabla\tilde{A}(\reg{u})}^2 dxdt\\
&\quad + \frac{1}{2}\norm{\nabla \reg{\K}}^2_{L^\infty}\norm{\tilde{u}}^2_{L^\infty(0,T;L^1)} \int_0^T\int \abs{\reg{u}}^2 dx.
\end{align*}
Note also, we can apply the chain rule (see Lemma 14 in \cite{BRB10} or Lemma 6 in \cite{BertozziSlepcev10}) and get
\begin{align*}
 -\int_0^T\int\langle \reg{u}_t(t), \tilde{A} \rangle dxds \leq \liminf_{t\rightarrow 0} \int \int_0^{\reg{u}}\tilde{A}(s) ds dx.
\end{align*} 

Furthermore, since $A$ is increasing, we have $\int_0^u \reg{A}(s)ds\leq \reg{A}(u)u$. 
Therefore,
\begin{align*}
\frac{1}{2}\int_0^T \int_{\abs{x} \leq R/2} \abs{\grad \reg{A}(u)}^2(x) dx & \leq A(\norm{\reg{u}(0)}_{L^\infty})\norm{\reg{u}(0)}_{L^1} + \frac{1}{2}\norm{\nabla \reg{\K}}^2_{L^\infty}\norm{\tilde{u}}^2_{L^\infty(0,T;L^1)}\norm{\reg{u}}_{L^2(0,T;L^2)}^2 \\ & \quad + \int_0^T \int_{R/2 \leq \abs{x} \leq R} -\grad \reg{A} \cdot \grad (\reg{A}\eta_R) + \frac{1}{2}\abs{\grad (\reg{A}\eta_R)}^2 dx.  
\end{align*}
 
The latter error term can be bounded as follows 
\begin{align*}
\int_{R/2 \leq \abs{x} \leq R} -\grad \reg{A} \cdot \grad (\reg{A}\eta_R) + \frac{1}{2}\abs{\grad (\reg{A}\eta_R)}^2 dx & \leq \frac{1}{2}\int_{R/2 \leq \abs{x} \leq R}\abs{\reg{A}(\reg{u})}^2\abs{\grad \eta_R}^2 dx \\
& \lesssim R^{-2}\reg{A}(\norm{\reg{u}}_\infty)\norm{\reg{A}(\reg{u})}_1. 
\end{align*}

Therefore, since $\reg{A}(\reg{u}) \in L^1 \cap L^\infty$, we have for all $T < \infty$, by taking $R \rightarrow \infty$, 
\begin{align*}
\norm{\nabla \reg{A}(\reg{u})}_{L^2(0,T,L^2)}^2 & \lesssim \tilde{A}(\norm{\reg{u}}_{L^\infty})\norm{\reg{u}}_{L^1} + \norm{\nabla \reg{\K}}^2_{L^\infty}\norm{\widetilde{u}}^2_{L^\infty(0,T;L^1)} \norm{\reg{u}}_2^2. 
\end{align*}

This completes the proof of \textit{(iii)}. 
This bound, along with the fact that $a'_\epsilon \geq \epsilon$, gives us the bound in \textit{(iv)}. \\

The bound on the second moment of $\reg{u}$, \textit{(v)}, follows from the bound on the $L^1$ norm of $\reg{A}(u)$:
\begin{align*}
\frac{d}{dt}\int \abs{x}^2 \reg{u}\;dx & = \int \abs{x}^2\Delta\reg{A}(\reg{u})\;dx - \int \abs{x}^2\nabla\cdot(\reg{u}\nabla \K\ast\reg{u})\;dx\\
&\leq 2\int x\cdot \nabla \reg{A}(\reg{u}) \;dx + 2 \int x\cdot \reg{u}\nabla \K\ast \reg{u}\;dx\\  
&\leq 2d\int \reg{A}(\reg{u}) dx + 2\norm{\nabla \reg{\K}\ast \widetilde{u}}_\infty \norm{\reg{u}}_{L^1}^{1/2}\left(\int \abs{x}^2 \reg{u}\; dx\right)^{1/2}\\
& \leq 2d\int \reg{A}(\reg{u}) dx + 2M^{1/2}\norm{\nabla\reg\K}_{L^\infty}\norm{\widetilde{u}}_{L^1}\left(\int \abs{x}^2 \reg{u}\; dx\right)^{1/2} \\
& \leq 2d\int \reg{A}(\reg{u}) dx + 2M^{1/2}\norm{\nabla\reg\K}_{L^\infty}\norm{\widetilde{u}}_{L^1}\left(1 + \int \abs{x}^2 \reg{u}\; dx\right). 
\end{align*}
An application of Gr\"onwall's inequality then provides the quantitative result. 
Let $\phi \in C([0,T];H^1)$, then by the definition of the weak solution (Definition \ref{def:WSRD}, Remark \ref{rmk:TimeWSRD}) and Cauchy-Schwarz,
\begin{align*} %\label{H1H-1Est}
\int_0^T \left\langle \partial_t\reg{u}(t), \phi(t)\right\rangle^2 dt \leq \int_0^T\norm{\nabla A(\reg{u})-\reg{u}\nabla\reg\kernel\ast \widetilde{u}}_{L^2(\Real^d)}\norm{\nabla \phi(t)}_{L^2(\Real^d)}dt. 
\end{align*}
Therefore the bounds we have already obtained can be combined to imply $\norm{\partial_t \reg{u} }_{L^2(0,T;H^{-1})} \leq C(T,\tilde{u},u_0)$.
\end{proof}
Now that we have all the required a priori estimates we are ready to prove Proposition \ref{prop:LERADD}.  
\begin{proof} (\textbf{Proposition \ref{prop:LERADD}})
Define the compact and convex subset of $L^2(0,T;L^2)$
\begin{align*}
\mathcal{S}_T = \set{v \in L^2(0,T; V) : \norm{v_t}_{L^2(0,T; H^{-1})} + \norm{v}_{L^2(0,T;V)}\leq C_0 \; \norm{v}_{L^\infty(0,T;L^1)} \leq \norm{u_0}_1}, 
\end{align*}  
for some $C_0$ to be chosen below.  
Compactness in the $L^2(0,T;L^2)$ topology follows from the Lions-Aubin lemma (since $V \subset \subset L^2$) and the fact that $\mathcal{S}_T$ is 
closed due to the weak compactness of $V$ and $L^2(0,T;H^{-1})$, and the lower semi-continuity of the $L^1$ norm. \\

Define the map $J:\mathcal{S}_T \rightarrow \mathcal{S}_T$ as follows: given $\tilde{u}\in \mathcal{S}_T$ we define $J\tilde{u} = u$, 
where $u$ is the solution to \eqref{eq:LinearProblem}.  
Our goal is to apply the Schauder fixed point theorem. 
First, we verify that $\mathcal{S}_T$ is invariant under $J$.\\
\\
\textit{Step 1:}  (Invariance of $\mathcal{S}_T$)  Let $\tilde{u} \in \mathcal{S}_T$. 
Recall from clasical regularity theory (e.g. \cite{Lieberman96} or \cite{Evans}) that since $u$ satisfies the a priori estimate \textit{(ii)} from Lemma \ref{lem:AprioriBoundLP} it is a global strong solution (of course it is not necessarily classical due to the potential irregularity of $\tilde{u}$). 
Furthermore, the additional bounds provided by Lemma \ref{lem:AprioriBoundLP} plus conservation of mass give that $u\in \mathcal{S}_T$ for $C_0$ sufficiently 
large and $T$ sufficiently small.  Note that $T$ and $C_0$ depend only on $M$, the $L^\infty$ norm of the initial data, and $\epsilon$.  Indeed, \textit{(i)} 
provides a bound on $\norm{u}_{L^2(0,T;L^2)}$, \textit{(iv)} provides a bound on $\norm{\grad u}_{L^2(0,T;L^2)}$, and \textit{(vi)} provides
a bound on $\norm{\partial_t u}_{L^2(0,T;H^{-1})}$. Moreover, the bound \textit{(ii)} and \textit{(v)} along with,
\begin{align*}
\int \abs{x}u^2dx \leq \norm{u}_\infty\left(\int \abs{x}^2 u dx\right)^{1/2}\norm{u}_1^{1/2}, 
\end{align*}
provides an estimate on the first moment of the $L^2$ norm of $u$. 
Hence, $J:\mathcal{S}_T \rightarrow \mathcal{S}_T$.  We are left to show that $J$ is a continuous map in $L^2(0,T;L^2)$.\\
\\
\textit{Step 2:} (Continuity in $L^2(0,T;L^2)$)  We show $J$ is continuous as a mapping from $L^2(0,T;L^2)$ to $C([0,T];\dot{H}^{-1})$, which by 
interpolation against uniform bounds in $H^1$ provided by \textit{(iv)}, implies continuity in $L^2(0,T,L^2(\Real^d)$ 
(since $\norm{f}_2 \leq \norm{\grad f}_{2}^{1/2}\norm{f}_{\dot{H}^{-1}}^{1/2}$ as can be seen from the Fourier transform).
This approach, as opposed to working in $L^2$ directly, is convenient due to the nonlinear diffusion, which behaves most naturally in $\dot{H}^{-1}$.
This argument is reminiscent of the $\dot{H}^{-1}$ stability estimate used to prove uniqueness of weak solutions \cite{BertozziSlepcev10,BertozziBrandman10,BRB10,AzzamBedrossian11}. 
Let $\set{v_n}_{n \geq 0} \subset \mathcal{S}_T$ be such that $v_n \rightarrow v$ in $L^2(0,T;L^2(\Real^d))$. 
We show that $Jv_n \rightarrow Jv$ in $C([0,T];\dot{H}^{-1}(\Real^d))$.
To this end, let $\phi_n := -\N \ast (Jv_n - Jv)$ and denote $u_n := Jv_n$ and $u := Jv$.    
It is important to note that while the $v_n$'s may not have constant $L^1$ norm, $u_n$ and $u$ do since they have the same initial data.\\

In order for us to proceed, we must show $-\Delta \phi = u_n - u$, in the sense of distributions, and 
that $\norm{\grad \phi_n}_\infty + \norm{\grad \phi_n}_2 \lesssim_{C_0, M,u_0,T} 1$.
To prove the former, it suffices to show that $\phi(t) \in L^\infty(\Real^d)$, which implies $-\Delta \phi$ is a bounded distribution. 
All of these are proven by arguments found in the proof of Theorem 1 in \cite{AzzamBedrossian11} so we only briefly recall them here.   
In dimensions $d \geq 3$, these facts can all be established using the $L^p$ estimates on $u_n$ and $u$ combined with Young's inequality and the fact that $\N \in L^{d/(d-2),\infty}$ and $\grad\N \in L^{d/(d-1),\infty}$. 
However, in $d = 2$ the problem is more delicate as $\N(x) = \frac{1}{2\pi} \log \abs{x}$ grows at infinity. 
Lemma \ref{lem:Hardy} implies $\phi(t) \in L^\infty$, as $\int u_n - u dx = 0$ and $u_n,u$ have uniformly bounded first moments. 
Proving $\norm{\grad \phi(t)}_2 \lesssim 1$ can be shown using the Fourier transform: $\int u_n - u dx = 0$ implies $\widehat{u}_n - \widehat{u}(0,t) \equiv 0$ and the uniform boundedness of the first moments of $u_n,u$ implies Lipschitz continuity of $\widehat{u}_n - \widehat{u}$. 
See \cite{AzzamBedrossian11} for more information. 
Now we may compute the evolution of $\norm{\grad \phi}_2^2 = \norm{u_n - u}_{\dot{H}^{-1}}^2$: 
\begin{align*}
\frac{1}{2}\frac{d}{dt} \int \abs{\grad \phi_n(t)}^2 dx & = <\grad \phi_n(t), \partial_t \grad \phi_n(t) >
= -<\phi(t), \partial_tu_n(t) - \partial_tu(t)>. 
\end{align*}
Therefore, using $\phi(t)$ in the definition of weak solution, 
\begin{align*}
\frac{1}{2} \frac{d}{dt} \int \abs{\grad \phi_n(t) }^2 dx & = -\int (\grad A^\epsilon(u_n) - \grad A^{\epsilon}(u)) \cdot \grad\phi_n dx \\
& -\int (u_n-u)(\conv{\grad\reg{\K}}{v})\cdot \grad \phi_n dx
 - \int u (\conv{\grad\reg{\K}}{(v_n-v)}) \cdot \grad \phi_n dx. \\
& := I_1 + I_2 + I_3. 
\end{align*}
We have dropped the time dependence for notational simplicity. 
Since $\reg{A}$ is increasing, we have the desired monotonicity of the diffusion,
\begin{equation*}
I_1 = - \int \left(\reg{A}(u_n) - \reg{A}(u)\right)(u_n-u) dx \leq 0.
\end{equation*}
Using integration by parts we have,
\begin{align*}
I_2 & = \int \Delta \phi_n \grad\reg{\K}\ast v\cdot \grad \phi_n dx \\
& = -\sum_{i,j}\int \partial_{x_j}\phi_n \partial_{x_j x_i}\reg{\K}\ast v \partial_{x_i} \phi_n dx -\sum_{i,j}\int \partial_{x_j}\phi_n \partial_{x_i}\reg{\K}\ast v \partial_{x_j x_i} \phi_n dx \\ 
& = -\sum_{i,j}\int \partial_{x_j}\phi_n \partial_{x_j x_i}\reg{\K}\ast v \partial_{x_i} \phi_n dx -\sum_{i,j}\frac{1}{2}\partial_{x_i}\left( (\partial_{x_j}\phi_n)^2 \right) \partial_{x_i}\reg{\K}\ast v dx \\ 
& = -\sum_{i,j}\int \partial_{x_j}\phi_n \partial_{x_j x_i}\reg{\K}\ast v \partial_{x_i} \phi_n dx +\sum_{i,j}\frac{1}{2}(\partial_{x_j}\phi_n)^2 \partial_{x_i x_i}\reg{\K}\ast v dx.
\end{align*}
Hence, 
\begin{align*}
I_2 & \lesssim \int \abs{D^2\reg{\K} \ast v} \abs{\grad \phi_n}^2 dx \\ 
& \leq \norm{D^2\reg{\K}}_2\norm{v}_2\int \abs{\grad \phi_n}^2 dx. 
\end{align*}
Moreover, 
\begin{align*}
\abs{I_3} & \leq \norm{u}_2\norm{\grad \reg{\K} \ast (v_n - v)}_\infty \norm{\grad \phi_n}_2 \\ 
& \lesssim \norm{\grad \reg{\K}}_3\norm{v_n - v}_{3/2}\norm{\grad \phi_n}_2 \\ 
& \lesssim_\epsilon \norm{v_n-v}_{2}^{2/3}\norm{v_n - v}^{1/3}_1 (1 + \norm{\grad \phi_n}_2^2). 
\end{align*}
By the uniform bound on $\norm{u}_\infty$ (by part \textit{(ii)} above), the uniform bound on the mass in $\mathcal{S}_T$ and the regularization of $\K$,   
\begin{equation*}
\frac{1}{2} \frac{d}{dt}\norm{\grad\phi_n}_2^2 \lesssim_\epsilon \norm{v_n - v}^{2/3}_2 + \left(1 + \norm{v}_2^2 + \norm{v_n - v}^2_2 \right)\norm{\grad \phi_n}_2^2. 
\end{equation*}
Integrating implies for some $C > 0$ depending on the uniform bounds and $\epsilon$ (using $\phi(0) \equiv 0$),  
\begin{equation*}
\norm{\grad \phi_n(t)}_2^2 \leq \int_0^t \exp\left(C\int_s^t 1 + \norm{v(t^\prime)}^2_2 + \norm{v_n(t^\prime) - v(t^\prime)}_2^2 dt^\prime \right) \norm{v_n(s) - v(s)}_2^{2/3}ds. 
\end{equation*}
Since $t \leq T$ we have, 
\begin{equation*}
\norm{\grad \phi_n(t)}_2^2 \lesssim e^{CT +C\norm{v}_{L^2(0,T;L^2)}^2 + C\norm{v_n - v}_{L^2(0,T;L^2)}^2}\int_0^T\norm{v_n(s) - v(s)}_2^{2/3}ds. 
\end{equation*}

By assumption, $\norm{v_n(s) - v(s)}_2 \rightarrow 0$ in $L^2((0,T))$ and hence also pointwise a.e. on $(0,T)$. Since $\norm{v_n(s) -v(s)}^{2/3}_2 \leq 1 + \norm{v_n(s) - v(s)}_2^2$, 
by the dominated convergence theorem we have $\norm{\grad \phi_n}_2 \rightarrow 0$ uniformly on $(0,T)$. 
Therefore, on $\mathcal{S}_T$, $J$ is a continous mapping from $L^2(0,T;L^2)$ to $C([0,T];\dot{H}^{-1})$ and hence by interpolation against the uniform $H^1$ bounds provided by \textit{(iv)} in Lemma \ref{lem:AprioriBoundLP}, is also a continuous mapping from $L^2(0,T;L^2)$ to $L^2(0,T;L^2)$. \\ 

Finally, we apply the Schauder fixed point theorem, which implies there exists a solution $Ju = u$ with $u \in \mathcal{S}_T$ for some $T > 0$.
By the regularization of $\K$, it is straightforward to extend this solution to the regularized system indefinitely and to use a boot-strap argument to show that $u$ 
is a classical solution to \eqref{def:RAE}.
 \end{proof}

Proposition \ref{prop:LERADD} provides a global family of classical solutions $\set{u}_{\epsilon>0}$;
however, we need to prove a priori bounds that are independent of the regularizing parameter $\epsilon$ to prove obtain any compactness and pass to a solution of the original system \eqref{def:ADD}.  
The proof of compactness will strongly depend on the following iteration result. 

\begin{lemma}[Iteration Lemma \cite{Kowalczyk05,CalvezCarrillo06}] \label{dynamic}
Let $0 < T\leq \infty$ and assume that there exists a $c>0$ and $u_c>0$ such that $A'(u)>c$ for all $u>u_c$. Then if 
$\left\|\nabla \kernel\ast u\right\|_{\infty} \leq C_1$ on $[0,T]$ then $\left\|u\right\|_\infty \leq C_2(C_1)\max\{1,M,\left\|u_0\right\|_\infty\}$ on the same time interval. 
\end{lemma}
    
The primary a priori bound required to obtain compactness is a bound on the $L^\infty$ norm which is independent of $\epsilon$. The proof of Lemma 8 in \cite{BRB10} 
applies directly to the following lemma; however, we include the proof for the completeness. 
In this setting we simplify the proof by applying a homogeneous Gagliardo-Nirenberg inequality, Lemma \ref{lem:GNS}. \\

\begin{lemma}[$L^\infty$ Bound of Solution] \label{LinftyG}
Let $\set{u^\epsilon}_{\epsilon > 0}$ be the classical solutions to \eqref{def:RAE} on $\pd$, with smooth, non-negative, and bounded initial data $\mollify{u_0}$. 
Then there exists $C = C(\norm{u_0}_1,\norm{u_0}_\infty)$ and $T = T(\norm{u_0}_1,\norm{u_0}_p)$ for any $p > d$ such that for all $\epsilon > 0$, 
\begin{align}\label{Linfty}
\left\|u^\epsilon(t)\right\|_{L^\infty(\Real^d)}\leq C
\end{align}
\noindent for all $t\in[0,T]$.
\end{lemma}
\begin{proof} 

For notational simplicity we drop the $\epsilon$.  The first step is to obtain an interval for which the $L^p$ norm of $u$ is bounded for some $p > d$.  Following the techniques commonly used, see for example 
\cite{JagerLuckhaus92,Kowalczyk05,CalvezCarrillo06,Blanchet09,BRB10}, we define the function $\reg{u}_{k}=(\reg{u}-k)_+$, for $k>0$.  Due to conservation of mass the following inequality 
provides a bound for the $L^p$ norm of $u$ given a bound on the $L^p$ norm of $u_k$,  
\begin{align}\label{UkvsU}
\norm{u}_p^p \leq C(p)(k^{p-1}\norm{u}_1 + \norm{u_k}_p^p).  
\end{align}  
We look at the time evolution of $\left\|u_k\right\|_{p}$. \\
\\
{\it Step 1:}  
\begin{align*}
\frac{d}{dt}\left\|u_k\right\|_p^p &= p\int u_k^{p-1}\nabla \cdot\left(\nabla A(u) - u\nabla \reg{\kernel}\ast u\right) dx \\
& = -p(p-1)\int A'(u) \nabla u_k \cdot\nabla u dx - p(p-1)\int uu^{p-2}_k\nabla u_k\cdot\nabla \reg{\kernel}\ast u\; dx.  \\ 
&\leq  -\frac{4(p-1)}{p} \int A'(u)\left|\nabla u_k^{p/2}\right|^2dx + p(p-1)\int u_k^{p-1}\nabla u_k \cdot \nabla \reg{\kernel}\ast u\;dx\\
&\quad + kp(p-1)\int u_k^{p-2}\nabla u_k\cdot \nabla \reg{\kernel}\ast u\;dx,
\end{align*}
where we used the fact that for $l>0$
\begin{align}\label{UeqUk}
u(u_k)^l = (u_k)^{l+1} + ku_k^l.  
\end{align}
\noindent Integrating by parts once more and using Lemma \ref{lem:PropAdmissible} $(b)$ gives, 
\begin{align*}
\frac{d}{dt}\left\|u_k\right\|^p_p & \leq -C(p) \int A'(u)\left|\nabla u_k^{p/2}\right|^2dx - (p-1) \int u_k^p \Delta \reg{\K}\ast u dx- kp\int u_k^{p-1} \Delta \reg{\K}\ast u dx\\
& \leq -C(p)\hspace{-5pt}\int \hspace{-5pt} A'(u)\left|\nabla u_k^{p/2}\right|^2dx + C(p)\left\|u_k\right\|_{p+1}^{p}\left\|\Delta \reg{\K} \ast u\right\|_{p+1} + C(p) k\left\|u_k\right\|_{p}^{p-1}\left\|\Delta \reg{\K} \ast u\right\|_{p} \\
& \leq -C(p) \hspace{-5pt} \int \hspace{-5pt}A'(u)\left|\nabla u_k^{p/2}\right|^2dx + C(p)\left(\left\|u_k\right\|_{p+1}^{p+1} + \left\|u\right\|^{p+1}_{p+1}\right) + C(p)k\left(\left\|u_k\right\|_{p}^{p} + \left\|u\right\|^p_{p}\right).
\end{align*}
Now, using \eqref{UkvsU} we obtain that
\begin{align*}
\frac{d}{dt}\left\|u_k\right\|_p^p dx &\leq -C(p)\hspace{-5pt}\int \hspace{-5pt} A'(u)\left|\nabla u_k^{p/2}\right|^2dx + C(p)\left\|u_k\right\|_{p+1}^{p+1} + C(p,k)\left\|u_k\right\|_{p}^{p} + C(p,k,M).
\end{align*} 
An application of the Homogeneous Gagliardo-Nirenberg inequality gives that for any $p$ such that $d<2(p+1)$ (e.g. Lemma \ref{lem:GNS} in the Appendix):
\begin{align*}
\left\|u\right\|_{p+1}^{p+1}\lesssim \left\|u\right\|^{\alpha_2}_{p}\left\|\nabla u^{p/2}\right\|^{\alpha_1}_{L^2(\Real^d)},
\end{align*}
where $\alpha_1 = d/p,\; \alpha_2 = (p+1)-d/2$.  From the the Young's inequality $a^rb^{(1-r)} \leq ra + (1-r)b$ (with
$a= \delta \left\|\nabla u^{p/2}\right\|^2_{L^2(\Real^d)}$ and $r=\alpha_1/2$) we obtain  
\begin{align*}
\left\|u\right\|_{p+1}^{p+1}&\lesssim \frac{1}{\delta^{\beta_1}}\left\|u\right\|^{\beta_2}_{p} + r\delta\left\|\nabla u^{p/2}\right\|^2_2.  
\end{align*}
Above $\beta_1,\;\beta_2 > 1$.  For $k$ large enough we have that $A'(u)>c>0$ over $\set{u > k}$; hence, if we choose $\delta$ small enough we obtain the final differential inequality:
\begin{align}\label{DILp}
\frac{d}{dt}\left\|u\right\|^p_p \lesssim C(p,\delta)\left\|u_k\right\|^{\beta_2}_p + C(p,k,r\delta) \left\|u_k\right\|_{p}^{p} + C(p,k,\left\|u_0\right\|_{1}).
\end{align}
The inequality \eqref{DILp} in turns gives a $T_p=T(p)>0$ such that $\left\|u_k\right\|_p$ is bounded on $[0,T_p]$ and by \eqref{UkvsU} then 
$\norm{u}_p$ remains bounded on the same time interval.  \\

\noindent {\it Step 2:}  We are not done since the bounds in \eqref{DILp} blow up as $p \rightarrow \infty$. 
We follow the by now standard procedure and prove that the velocity field is bounded in $L^\infty(\Real^d)$ on some time interval $[0,T]$.  
Then this allows us to invoke Lemma \ref{dynamic} and obtain the desired bound.
Since $\grad K \in L_{loc}^1$ and $\grad K \mathbf{1}_{\Real^d \setminus B_1(0)} \in L^q$ for all $q > d/(d-1)$ (by Lemma \ref{lem:PropAdmissible} $(a)$), then for all $p > d/(d-1)$ 
\begin{equation*}
\norm{\conv{\grad \K}{u}}_p \leq \norm{\grad \K \mathbf{1}_{B_1(0)}}_1\norm{u}_p + \norm{\grad \K \mathbf{1}_{\Real^d\setminus B_1(0)}}_pM.
%\norm{\vec{v}}_{p} = \norm{\conv{\grad \K}{u}}_{L^{p}(B_1(x_0) \cap \D)} \lesssim \norm{u}_{p}.
\end{equation*}
By Lemma \ref{lem:PropAdmissible} (b) we also have, for all $p$, $1 < p < \infty$, 
\begin{equation*}
\norm{\conv{D^2\K}{u}}_p \lesssim \norm{u}_p. 
\end{equation*} 
Then by Morrey's inequality we have $\grad \K \ast u^\epsilon \in L^\infty(\Real^d)$ uniformly in $\epsilon$ by choosing some $p > d$ and invoking step one.  Then Lemma \ref{dynamic} concludes the proof. 
\end{proof} 

\begin{proof} (Theorem \ref{thm:loc_theory})
For all $\epsilon > 0$, let $u^\epsilon$ be the solution to \ref{def:RAE} provided by Proposition \ref{prop:LERADD}. 
Lemma \ref{LinftyG} provides a uniform-in-$\epsilon$ a priori upper bound on $\norm{u^\epsilon(t)}_\infty$ on some time inverval $(0,T)$ as well as a uniform bound on $\norm{\grad \K \ast u}_\infty$.  
Hence, we may easily deduce the following a priori bound on the second moment: 
\begin{align*}
\frac{d}{dt}\mathcal{M}_2(\reg{u}(t)) & = 2d\int \reg{A}(\reg{u}) dx - 2\int\int \reg{u} x \cdot \grad \reg{\K} \ast \reg{u} dx \\
 & \leq 2d\int \reg{A}(\reg{u}) dx + 2\norm{\grad \reg{\K} \ast \reg{u}}_\infty \norm{u}_1^{1/2} \mathcal{M}_2(\reg{u}(t))^{1/2} \\
& \lesssim_{\norm{u}_1} \int \reg{A}(\reg{u}) dx + \left(\norm{\reg{u}}_\infty + \norm{u}_1\right)\left(1 + \mathcal{M}_2(\reg{u}(t)) \right). 
\end{align*}
The first term is bounded in terms of $\norm{u}_1$ and $\norm{u}_\infty$ as in the proof of Lemma \ref{lem:AprioriBoundLP}.
By Gr\"onwall's inequality we therefore have a uniform-in-$\epsilon$ a priori upper bound on the second moment.
Similarly, one may alter the proof of \textit{(iii)} in \ref{lem:AprioriBoundLP} to bound $\norm{\grad A}_{L^2(0,T;L^2)}$ independent of $\epsilon$ using the uniform bound on $\norm{\reg{u}}_\infty$.
Using these a priori bounds we may follow the arguments of \cite{BRB10} (see also \cite{BertozziSlepcev10}) and prove the following lemma. 
\begin{lemma}[Local Pre-compactness in $L^1(\Real^d_T)$]\label{precomp}   
The sequence of solutions obtained via Proposition \ref{prop:LERADD}, $\{u_\epsilon\}_{\epsilon > 0}$, is pre-compact in $L^1((0,T) \times \Real^d)$.  
\end{lemma}

Equi-continuity relies on the fact that $\norm{A(u^\epsilon)}_{L^2(0,T;H^1(D))} \leq C$ uniformly in $\epsilon$.  
The proof depends on the domain size but when combined with tightness obtained from a priori bounds on the second moment (see below), this suffices.  
Hence, using Lemma \ref{precomp} we may extract a subsequence $u_{\epsilon_k} \rightarrow u$ in $L^1((0,T) \times \Real^d)$. 
Additionally, since $\norm{u_{\epsilon_k}}_\infty$ is uniformly bounded, by lower-semicontinuity, we also have $u \in L^\infty((0,T)\times \Real^d)$. 
Once the limit is extracted, we may further upgrade the convergence to $C([0,T];L^1)$ as in \cite{BertozziSlepcev10,BRB10} and hence to $C([0,T];L^p)$ for all $1 \leq p < \infty$ due to the uniform $L^\infty$ bound. 
The last remaining technical point is to ensure that the limit is indeed a weak solution of the original system in the sense of Definition \ref{def:WSRD}.
This may be done using a limiting argument and the a priori bounds available on $u$ as in \cite{BRB10,BertozziSlepcev10}. \\

Proving the energy dissipation inequality in $\Real^2$ also presents an additional complication. 
As $\K$ is potentially logarithmically unbounded at infinity and hence the convergence of the interaction energy is non-trivial, we cannot follow the exact argument from \cite{BRB10}.
Hence, let $u(t)$ be the unique weak solution constructed above and $\set{u^\epsilon}_{\epsilon > 0}$ be the solutions 
to the regularized system \eqref{def:RAE}. 
We may follow the proof in \cite{BRB10} to prove convergence of the entropy and the energy dissipation,
hence, we need only focus on the interaction energy.  
To this end,
\begin{align*}
\int \reg{u}(t) \reg{\K} \ast \reg{u}(t) dx - \int u(t) \K \ast u(t) dx & = \int u(\reg{\K} - \K)\ast u + \int u\reg{\K}\ast(\reg{u} - u)dx \\ & \quad + \int(\reg{u} - u)\reg{\K} \ast \reg{u} dx \\
& := T1 + T2 + T3.  
\end{align*} 
Since $\reg{\K}(x) = \K(x)$ for all $\abs{x} > 1$ we have, 
\begin{align*}
\abs{T1} &\leq \norm{u}_1\norm{(\reg{\K} - \K) \ast u}_\infty \\
& \leq \norm{u}_1\norm{u}_\infty \norm{\reg{\K} - \K}_{L^1(B_1(0))} \rightarrow 0.   
\end{align*}
Now consider $T2$. By the duality of $BMO$ and $\mathcal{H}^1$ and Lemma \ref{lem:Hardy} we have,   
\begin{align*}
\abs{T2} & \leq \norm{u}_1\norm{\reg{\K} \ast (\reg{u} - u)}_\infty \\
& \lesssim \norm{\reg{\K}}_{BMO}\norm{\reg{u} - u}_{\mathcal{H}^1} \\
& \lesssim \norm{\reg{u} - u}_p + \int \abs{x}\abs{\reg{u} - u} dx, 
\end{align*}
for any $1 < p < \infty$. However, $\reg{u} \rightarrow u$ in $C([0,T];L^p)$ for all such $p$, so the first term is not an issue. 
To deal with the second term we first use Cauchy-Schwarz 
\begin{equation*}
\int \abs{x}\abs{\reg{u} - u} dx \leq \left( \int \abs{x}^2\abs{\reg{u} - u} dx \right)^{1/2}\left(\int \abs{\reg{u} - u} dx \right)^{1/2}.  
\end{equation*}
However, since $\reg{u} \rightarrow u$ in $C([0,T];L^1)$ and both $\reg{u}$ and $u$ have uniformly bounded second moments on $[0,T]$ we have that $T2 \rightarrow 0$. 
The final term, $T3$, follows similarly. Hence, the energy dissipation inequality holds in $\Real^2$. 
\end{proof}
%%%%%%%%%%%%%%%%%%%%%%%%%%%%%%%%%%%%%%%%%%%%%%%%%%%%%%%%%%%%%%%%%%%%%
\subsection{Local Existence for Spatially Inhomogeneous Patlak-Keller-Segel Systems} 
In this section we prove local existence of solutions of \eqref{def:InhKS}. Once again, we begin with a regularization, which in this case is given by
\begin{equation} \label{def:RegdInhKS}
\left\{
\begin{array}{l}
  u^\epsilon_t + \nabla \cdot (u^\epsilon \nabla \mollify\reg{c}) = \Delta A^\epsilon(u^\epsilon), \\
  -\grad \cdot (a(x)\nabla \reg{c}) + \gamma(x)\reg{c} = \mollify u^\epsilon, \\ 
  u^\epsilon(0,x) = \mollify{u_0}(x), \;\; u_0 \in L_+^1(\Real^d;(1+\abs{x}^2)dx)\cap L^\infty(\Real^d),
\end{array}
\right.
\end{equation}
where $A^\epsilon(\reg{u})$ is the parabolic regularization given by \eqref{ParabolicReg}.
This particular regularization is convenient because \eqref{def:RegdInhKS} satisfies a natural energy dissipation inequality,
\begin{align*}
\F^\epsilon(u^\epsilon) & =\int \Phi^\epsilon(u^\epsilon) dx - \frac{1}{2}\int u^\epsilon \mollify{c^\epsilon} dx.  
\end{align*} 
Indeed, using 
\begin{align*}
\frac{1}{2}\frac{d}{dt}\int\reg{u}\mollify c \;dx = \int \reg{u}_t\mollify c \;dx,
\end{align*}
one can formally obtain 
\begin{align*}
\frac{d}{dt}\F^\epsilon(u^\epsilon)= -\int \frac{1}{\reg{u}}\abs{ (A^\epsilon)^\prime(u^\epsilon)\grad \reg{u} - \reg{u} \grad \mollify{c}}^2 \;dx.    
\end{align*} 
Moreover, the regularization of the chemo-attractant only occurs on the right-hand side, which proves to be of some technical use.

The proof of Theorem \ref{thm:loc_theory_variable} is very similar to that of Theorem \ref{thm:loc_theory}: one simply needs to check that $c$ satisfies properties sufficiently similar to those satisfied by convolution-type systems.  Most of these properties are obatained via standard elliptic estimates, which we state and prove in the Appendix to be referenced as needed.
As was done for the convolution-type systems, we first need equivalent estimates as those provided by Lemma \ref{lem:AprioriBoundLP}
for the system
\begin{equation} \label{def:RegdInhKS_lin}
\left\{
\begin{array}{l}
  u^\epsilon_t + \nabla \cdot (u^\epsilon \nabla \mollify{c}) = \Delta A^\epsilon(u^\epsilon), \\
  -\grad \cdot (a(x)\nabla c^\epsilon) + \gamma(x)c^{\epsilon} = \mollify{\tilde{u}}, 
\end{array}
\right.
\end{equation}
where $\tilde{u}\in L^2(0,T;L^2)\cap L^\infty(0,T;L^1)$ is given.  As above, this will allow us to obtain the necessary compactness and a priori estimates to apply the Schauder fixed point theorem
and obtain a family of classical solutions for \eqref{def:RegdInhKS_lin}, as in Proposition \ref{prop:LERADD}.     

\begin{proof}(\textbf{Theorem} \ref{thm:loc_theory_variable})
As this proof follows that of Theorem \eqref{thm:loc_theory} we simply mention and work out the differences.  The first step is to prove 
the existence of solutions to the regularized system \eqref{def:RegdInhKS_lin}.  This is done, as in the proof of Proposition \eqref{prop:LERADD}, 
by a fixed point argument.  Hence, we need to show that for a given $\tilde{u}$ the system \eqref{def:RegdInhKS_lin} has a solutions that
satisfies the appropriate bounds.  
Examining the proof of Lemma \eqref{lem:AprioriBoundLP} we see that all the bounds hold provided we have 
the bounds on $\norm{\Delta \mollify c}_\infty $ and $\norm{\nabla \mollify c}_\infty$.  Using Young's inequality for convolutions,
when $\gamma >0$ we obtain
\begin{align*}
\norm{\Delta \mollify{c}}_\infty & \lesssim_{\epsilon} \norm{c}_2 \lesssim_\gamma \norm{\tilde{u}}_2 \\
\norm{\grad \mollify{c}}_\infty  & \lesssim_{\epsilon}\norm{c}_2 \lesssim_\gamma \norm{\tilde{u}}_2.
\end{align*}

The estimates are different in the case when $\gamma \geq 0$, as we only have estimate \eqref{ineq:HomogCEstimate} (necessarily we are assuming $d \geq 3$).  We instead have, 
\begin{align*}
\norm{\Delta \mollify{c}}_\infty & \lesssim_{\epsilon} \norm{c}_{\frac{6d}{5d-12}} \lesssim \norm{\tilde{u}}_{6/5} \leq \norm{\tilde{u}}^{1/3}_2\norm{\tilde{u}}^{2/3}_1.
\end{align*}
Although this estimate involves a lower power of $\norm{\tilde{u}}_2$, this does not pose a problem for our needs. 
A similar bound for $\norm{\grad \mollify{c}}_\infty$ can be obtained.  
These bounds are sufficient to obtain a family of solutions on $\set{\reg{u}}_{\epsilon>0}$ on $[0,T_\epsilon)$ by following the proof
of Proposition \ref{prop:LERADD}.  Next, we need to prove a uniform-in-$\epsilon$ $L^\infty$ bound of the solutions, similar to Lemma \ref{LinftyG}.  
Recall that the first step in the proof of Lemma \ref{LinftyG} was to obtain a bound on $\norm{\reg{u}}_p$ for $p>1$.  This required a bound of the form
\begin{align}\label{ineq:general_cz}
\norm{\Delta c^\epsilon}_{p}\leq C(p,a,d)\norm{\reg{u}}_p,
\end{align}
for example, to obtain the inequalities after \eqref{UeqUk}. If $\gamma(x)$ is strictly positive, a proof of \eqref{ineq:general_cz} can be found in Lemma \ref{lem:inhom_bds_CZ} in the Appendix.  
For the remaining cases, $d\geq 3$, this estimate is replaced by
\begin{align*}
\norm{\Delta c^\epsilon}_{p}\leq C(p,a,d)\left(\norm{\reg{u}}_p +\norm{\reg{u}}_1\right),
\end{align*}
which is obtained from Lemma \ref{lem:HomogGrad}, noting that $\frac{pd}{2p+d}<p$ and interpolating between the $L^p$ and $L^1$ norms.  
The the next step was to use the bounds on the $L^p$ norms of $\reg{u}$ to deduce a bound on $\norm{\nabla \reg{c}}_\infty$, independent of $\epsilon$.  
By Morrey's inequality and then the Gagliardo-Nirenberg inequality \eqref{ineq:WeightSE2},
\begin{align*}
\norm{\grad c}_\infty & \lesssim \norm{D^2 c}_p + \norm{\grad c}_p \\
& \lesssim \norm{D^2c}_p + \norm{c}_p. 
\end{align*}
If $\gamma$ is strictly positive then we may use Lemma \ref{lem:inhom_bds_CZ} and we have 
\begin{equation*}
\norm{\grad c}_\infty \lesssim \norm{u}_p. 
\end{equation*}
If $\gamma$ is not strictly positive we have to use the $L^p$ estimate \eqref{ineq:HomogCEstimate} and Lemma \ref{lem:HomogGrad}, which holds in $d \geq 3$ for $p > d/(d-2)$: 
\begin{equation*}
\norm{\grad c}_\infty \lesssim_{a,p,d} \norm{u}_p + \norm{u}_{\frac{pd}{d+2}} \lesssim \norm{u}_p + \norm{u}_{1}. 
\end{equation*}
Hence, we may proceed as in Lemma \ref{LinftyG} and apply Lemma \ref{dynamic} and deduce the requisite local in time, uniform-in-$\epsilon$ $L^\infty$ estimates.

The last primary estimate we need before we can obtain enough pre-compactness to extract a subsequence which converges towards a solution to the original model \eqref{def:InhKS} is
a uniform estimate on $\norm{A(u)}_{L^2(0,T,H^1)}$.  Examining the proof of estimate $(iii)$ in Lemma \ref{lem:AprioriBoundLP}, one can see that the bound on
$\norm{\nabla c}_\infty$ is sufficient.  
Analogously as to above, pre-compactness of $\set{\reg{u}}_{\epsilon>0}$ follows from these uniform estimates. 
Finally, to obtain a solution to \eqref{def:InhKS} the only real difference with convolution-type systems is that we need to check that $\norm{\reg{u}\nabla \reg{c}-u\nabla c}_{L^2}\rightarrow 0$ as $\epsilon\rightarrow 0$.
Indeed, we know that
\begin{align*}
\norm{\reg{u}\nabla \reg{c}-u\nabla c}_{L^2} &\leq \norm{\reg{u}}_{L^\infty}\norm{\nabla(\reg{c}-c)}_{L^2} + \norm{\nabla c}_\infty\norm{\reg u - u}_{L^2}.  
\end{align*}  
 Furthermore, again using \eqref{ineq:WeightSE2} we have  
 \begin{align*}
\norm{\nabla(\reg{c}-c)}_{L^2}&\lesssim \norm{D^2(\reg{c}-c)}_{L^2} +\norm{\reg{c}-c}_{L^2}. 
 \end{align*}
The result follows from Lemma \ref{lem:inhom_bds_CZ} or Lemma \ref{lem:HomogGrad} and the convergence in $L^p$ of $\mollify{\reg{u}}$ to $u$.  
Hence, one may deduce that $u$ is a weak solution to \eqref{def:InhKS} in the sense of Definition \ref{def:WSRD}.  
We are left to prove the energy dissipation inequality \eqref{EnrDiss2}.  Convergence of the entropy term follows exactly as the proof of inequality \eqref{EnrDiss}, given in Theorem \ref{thm:loc_theory}.
The interaction energy on the other hand requires a little more. As above, 
\begin{align*}
\norm{\mollify{\reg{u}}\reg{c}-uc}_{L^1} \leq \norm{\mollify{\reg{u}}}_1\norm{\reg{c}-c}_{L^\infty} + \norm{\reg{c}}_\infty\norm{\mollify{\reg u} - u}_{L^1}.  
\end{align*}  
Given that $\reg{u}\rightarrow u$ in $L^1$ and that we have uniform bounds on $\norm{\reg{c}}_\infty$ by 
 and $\norm{\reg{u}}_{L^1}=M$ we only 
have to verify the uniform convergence of $\reg{c}$ to $c$.  
It follows from similar arguments as above, using Morrey's inequality and \eqref{ineq:WeightSE2} 
\begin{align*}
\norm{\reg{c}-c}_{L^\infty}&\lesssim \norm{\nabla(\reg{c}-c )}_p + \norm{\reg{c}-c}_p\\
&\lesssim \norm{D^2(\reg{c}-c )}_p +\norm{\reg{c}-c}_p, 
\end{align*}
Hence, the result follows from Lemma \ref{lem:HomogGrad} and \eqref{ineq:HomogCEstimate} or Lemma \ref{lem:inhom_bds_CZ} and the $L^p$ uniform-in-time convergence of $\reg{u}$ to $u$.   
The last step that requires attention is proving the convergence of the free energy dissipation (also called the generalized Fisher information), 
\begin{equation*}
\int_0^t D[u(s)]ds = \int_0^t\int \frac{1}{u(x,s)}\left|A'(u(x,s))\nabla u(x,s) -u(x,s)\grad c(x,s)\right|^2dxds. 
\end{equation*}
However, as in the convolution-type case in \cite{BertozziSlepcev10,BRB10}, it follows from a weak lower-semicontinuity result due to Otto \cite{Otto01} (see also \cite{CarrilloEntDiss01}).
\end{proof}

%%%%%%%%%%%%%%%%%%%%%%%%%%%%%%%%%%%%%%%%%%%%%%%%%%%%%%%%%%%%%%%%%%%%%%%%%%%%%%%%%%%%%%%%%%%%%%%
\section{Continuation, Subcritical and Small Data Theory} \label{sec:ContSubEtc} 
%%%%%%%%%%%%%%%%%%%%%%%%%%%%%%%%%%%%%%%%%%%%%%%%%%%%%%%%%%%%%%%%%%%%%%%%
\subsection{Continuation for Spatially Inhomogeneous Patlak-Keller-Segel Systems} \label{sec:ContInhKS}
The proof of Theorem \ref{thm:Continuation} is very similar to that of Lemma \ref{LinftyG} and largely follows the work of \cite{CalvezCarrillo06,Blanchet09}. 
Both of these results quantify the strength of the nonlocal advection term with the inequality 
\begin{equation*}
\norm{\Delta\K \ast u}_p \lesssim_{p,d} \norm{u}_p. 
\end{equation*}

For the system \eqref{def:InhKS}, unless $\gamma(x)$ is strictly positive, we only have the estimate provided by Lemma \ref{lem:HomogGrad} in the Appendix: for $p > d/(d-2)$, 
\begin{equation*}
\norm{\Delta c}_p \lesssim_{p,d} \norm{u}_p + \norm{u}_{\frac{pd}{2p+d}} \lesssim \norm{u}_p + \norm{u}_1.  
\end{equation*}  

The lower-order term (which is constant due to conservation of mass) can be safely added into other lower-order terms in the argument which are already present due to the inequality \eqref{UkvsU}.  
The difference between Lemma \ref{LinftyG} and Theorem \ref{thm:Continuation} is the use 
of a better Gagliardo-Nirenberg inequality to control the highest order contribution from the advection term: 
\begin{equation*}
\norm{u_k}^{p+1}_{p+1} \lesssim \norm{u_k}^{\alpha_2(p+1)}_{\max\left(1,\frac{2-m}{2-m^\star}\right)}\int \abs{\grad u_k^{ (p+m^\star-1)/2 }}^2 dx, 
\end{equation*}
with $\alpha_2 > 0$ determined via scale invariance and homogeneity (Lemma \ref{lem:GNS}). 
The latter factor can be related to the nonlinear diffusion and the former factor is the quantity appearing in \eqref{cond:equint}. The key point is that the condition \eqref{cond:equint} then implies we may make the highest order contribution of the nonlinear advection arbitrarily small compared to the diffusion by choosing $k$ large. 

%%%%%%%%%%%%%%%%%%%%%%%%%%%%%%%%%%%%%%%%%%%%%%%%%%%%%%%%%%%%%%%%%%%%%%%%
\subsection{Subcritical and Small Data Theory}
\begin{proof}(\textbf{Theorem \ref{thm:GBDs}}) 
We proceed with the formal computations, noting that the computations are completely rigorous by appealing to the parabolic regularization above. 
The proof follows a similar outline to the proof of Lemma \ref{LinftyG} or Theorem \ref{thm:Continuation}, which is now standard in the literature on PKS and related models. 
As remarked above, the subcritical case is a standard variant of the proof of Theorem \ref{thm:Continuation} \cite{Kowalczyk05,CalvezCarrillo06,BRB10}; therefore, we only prove \textit{(ii)}. 
We first prove the result for convolution-type systems and for variable-coefficient systems where $\gamma$ is strictly positive.  \\

Let $\overline{q} = \max((2-m)/(2-m^\star),1)$. We have two cases to consider: $m^\star = 2 - 2/d$ and $m^\star < 2-2/d$.
We make this distinction in order to take advantage of the additional regularity provided by \textit{(c)} of Lemma \ref{lem:PropAdmissible} when $m^*<2-2/d$; hence, in this case we let $\beta$ be such that  $m^\star = 1 + 1/\beta - 2/d$ for some $1 < \beta \leq d/2$. Note that by Lemma \ref{lem:PropAdmissible} $D^2\K \in L^{\beta,\infty}$. \\ 

\noindent \textit{Step 1 ($L^p$ bound):}  Choose $p$ such that if $m^*= 2-2/d$ then $\overline{q} \leq p < \infty$ else $\overline{q} \leq p \leq \beta/(\beta-1)$. Computing the time evolution of $\norm{u(t)}_p$ making use of \eqref{cond:Aprime_GBDs2} we obtain,
\begin{equation*}
\frac{d}{dt} \norm{u}_p^p \leq -4(p-1)\delta\int u^{m-1}\abs{\grad u^{p/2}}^2 dx - (p-1)\int u^p \Delta cdx. 
\end{equation*}
If $m^\star = 2-2/d$, we may use H\"older's inequality and then Lemma \ref{lem:PropAdmissible} for convolution-type systems or Lemma \ref{lem:inhom_bds_CZ} for variable coefficient systems to obtain, 
\begin{equation*}
\abs{\int u^p \Delta c dx} \lesssim_{p,\K} \norm{u}_{p+1}^{p+1}.
\end{equation*}
On the other hand, if the system is of convolution-type and $\beta > 1$, since $D^2\K \in L^{\beta,\infty}$ we can apply \eqref{ineq:GGHLS}, 
\begin{align*}
\abs{\int u^p \conv{\Delta \K}{u} dx}
& \lesssim \norm{u}^{p}_{\alpha p}\norm{u}_t\wknorm{\Delta\K}{\beta},
\end{align*}
with the scaling condition $1/\alpha + 1/t + 1/\beta = 2$. 
Choosing $t = \alpha p$ implies,
\begin{equation}
\frac{1}{\alpha} = \frac{2 - 1/\beta}{1 + 1/p}. \label{eq:alpha2}
\end{equation}
Notice that from our choice of $p$, $1 \leq 1/p + 1/\beta$; thus, $1/\alpha \leq 1$.  When $m^\star = 2 - 2/d$, we define $t = \alpha p = p+1$. 
By the Gagliardo-Nirenberg inequality Lemma \ref{lem:GNS}, using $m \leq m^\star$ we have, 
\begin{equation}
\norm{u}_{\alpha p} \lesssim \norm{u}^{\alpha_2}_{\overline{q}}\left(\int \abs{\grad u^{ (p+m-1)/2 }}^2 dx \right)^{\alpha_1/2} \label{ineq:GNS_Gbds2},
\end{equation}
with 
\begin{equation*}
\alpha_1 = \frac{2d}{p}\left(\frac{(p - \overline{q}/\alpha)}{ \overline{q}(2 - d) + dp + d(m-1)}\right),
\end{equation*}
and
\begin{equation*}
\alpha_2 = 1 - \alpha_1(p+m-1)/2.
\end{equation*}
Notice that by our choices, $\alpha_1(p+1)/2 = 1$. 
Hence \eqref{ineq:GNS_Gbds2} implies,
\begin{equation}
\frac{d}{dt} \norm{u}_p^p \leq \left(C(p)\norm{u}_{\overline{q}}^{\alpha_2(p+1)} - C(p)\delta\right)\int u^{m-1}\abs{\grad u^{p/2}}^2 dx. \label{ineq:gwp_2}
\end{equation} 
If $m = m^\star$ then $\overline{q} = 1$  and for mass sufficiently small we have $\norm{u(t)}_p \in L^{\infty}(\Real^+)$. 
If $m < m^\star$ then $\overline{q} > 1$ and $\norm{u(t)}_{\overline{q}}$ is not conserved. However, in that case \eqref{ineq:gwp_2} also holds for $p = \overline{q}$, and therefore if
$\norm{u_0}_{\overline{q}}$ is sufficiently small, \eqref{ineq:gwp_2} implies $\norm{u(t)}_{\overline{q}} \leq \norm{u_0}_{\overline{q}}$ for all $t > 0$. Hence it follows again from \eqref{ineq:gwp_2} that there exists a constant $C(p)$ such that if $\norm{u_0}_{\overline{q}} < C(p)$ then $\norm{u(t)}_p \in L^{\infty}(\Real^+)$. \\

\noindent \textit{Step 2 ($L^\infty$ bound):}
Once more we use Lemma \ref{dynamic}, which implies that if $\grad c \in L^\infty( (0,\infty) \times \Real^d)$ then $u \in L^\infty( (0,\infty) \times \Real^d)$. 
 If $m^\star = 2-2/d$, by Lemma \ref{lem:PropAdmissible} or Lemma \ref{lem:inhom_bds_CZ}, 
\begin{equation*}
\norm{D^2c}_{d+1} \lesssim \norm{u}_{d+1}. 
\end{equation*} 
For systems of convolution type we have by weak Young's inequality, 
\begin{equation*}
\norm{\grad \K \ast u}_{2d/(d-1)} \lesssim \norm{\grad \K}_{L^{d/(d-1),\infty}}\norm{u}_{(2d-1)/(d-1)}, 
\end{equation*}
and for variable coefficient systems we may use \eqref{ineq:WeightSE2} and Lemma \ref{lem:inhom_bds_CZ}. 
\begin{equation*}
\norm{\grad c}_{d+1} \lesssim \norm{D^2c}_{d+1} + \norm{c}_{d+1} \lesssim \norm{u}_{d+1}. 
\end{equation*} 
By Step 1, we may choose $\norm{u_0}_{\bar{q}}$ sufficiently small such that both of these norms are uniformly bounded, which is sufficient to apply Morrey's inequality and conclude that $\grad c \in L^{\infty}( (0,\infty) \times \Real^d)$.
In the case of convolution-type systems with $1 \leq m^\star < 2-2/d$ we may only bound $L^p$ norms of $u$ uniformly in time for $1 \leq p \leq \beta/(\beta -1)$.  
Since $\grad \K \in L_{loc}^1$ and $\grad \K \mathbf{1}_{\Real^d \setminus B_1(0)} \in L^q$ for all $q > d/(d-1)$ (by Lemma \ref{lem:PropAdmissible}), we have for any $q > d/(d-1)$
\begin{equation*}
\norm{\conv{\grad \K}{u}}_{q} \leq \norm{\grad \K \mathbf{1}_{B_1(0)}}_1\norm{u}_q + \norm{\grad \K \mathbf{1}_{\Real^d\setminus B_1(0)}}_qM.
\end{equation*}
Since $m^\star < 2-2/d$, $\beta > 1$, and then we may choose $q \in (d/(d-1),\beta/(\beta - 1)]$, since in this case necessarily $d \geq 3$. By the weak Young's inequality we have, 
\begin{equation*}
\norm{D^2\K \ast u}_{d+1} \lesssim \wknorm{D^2\K}{\beta}\norm{u}_{\frac{\beta(d+1)}{d(\beta - 1) + 2\beta - 1}},  
\end{equation*}
where
\begin{equation*}
1 < \frac{\beta(d+1)}{d(\beta - 1) + 2\beta - 1} \leq \frac{\beta}{\beta - 1}. 
\end{equation*}
By Step 1, we may again choose $\norm{u_0}_{\bar{q}}$ sufficiently small to ensure that $\norm{u}_{\frac{\beta(d+1)}{d(\beta - 1) + 2\beta - 1}}$ and $\norm{u}_q$ are both uniformly bounded.   
Therefore, we may again apply Morrey's inequality and deduce $\grad \K \ast u\in L^\infty( (0,\infty \times \Real^d)$ and conclude the proof using Lemma \ref{dynamic}.\\

In order to treat the case $d\geq 3$ and $\gamma(x)$ not strictly positive we use the less optimal estimate provided by Lemma \ref{lem:HomogGrad}. 
Hence beginning as before we have for $p \geq \bar{q}$ and $p>1$, 
\begin{equation*}
\frac{d}{dt} \norm{u}_p^p \leq -4(p-1)\delta\int u^{m-1}\abs{\grad u^{p/2}}^2 dx - (p-1)\int u^p \Delta cdx. 
\end{equation*}
Using Lemma \ref{lem:HomogGrad} we have, 
\begin{equation}
\abs{\int u^p \Delta c dx} \lesssim_{p} \norm{u}_{p+1}^{p+1} + \norm{u}_{p+1}^{p}\norm{u}_{\frac{(p+1)d}{2p+2+d}}, \label{ineq:lower_order_adv}
\end{equation}
note that necessarily $p + 1 > 2 > d/(d-2)$ and $\frac{(p+1)d}{2p+2+d} \leq p+1$. 
By the Gagliardo-Nirenberg inequality, Lemma \ref{lem:GNS},
\begin{equation}
\norm{u}_{p+1} \lesssim \norm{u}^{\alpha_2}_{\overline{q}}\left(\int \abs{\grad u^{ (p+m-1)/2 }}^2 dx \right)^{\alpha_1/2},
\end{equation}
with 
\begin{equation*}
\alpha_1 = \frac{2d}{p}\left(\frac{(p - p\overline{q}/(p+1))}{ \overline{q}(2 - d) + dp + d(m-1)}\right),
\end{equation*}
and
\begin{equation*}
\alpha_2 = 1 - \alpha_1(p+m-1)/2.
\end{equation*}
As above $\alpha_1(p+1)/2 = 1$. 
Let $\theta \in (0,1)$ such that $\norm{u}_{\frac{(p+1)d}{2p+2+d}} \leq M^{1-\theta}\norm{u}_{p+1}^\theta$. 
Therefore by \eqref{ineq:lower_order_adv} we have, 
\begin{align}
\frac{d}{dt} \norm{u}_p^p \leq \left(C(p) - \delta\frac{C(p,d)}{\norm{u}_{\bar{q}}^{\alpha_2(p+1)}} \right)\norm{u}_{p+1}^{p+1} + C(p)M^{p+1}. \label{ineq:LpEvoSmallData}
\end{align}
If $\bar{q} = 1$ we may choose $M$ sufficiently small (depending on $p$) so that the first term is negative (in fact, as negative as we require, say $< -\delta/2$). 
Since again by interpolation $\norm{u}_p^p \leq \norm{u}_{p+1}^{p+1} + M$ we then have 
\begin{align*}
\frac{d}{dt} \norm{u}_p^p \leq -\frac{\delta}{2}\norm{u}_p^p + C(p)M^{p+1} + \frac{\delta}{2}M. 
\end{align*}
Hence for all $p$ we may choose $M$ sufficiently small such that the $L^p$ norm of $u(t)$ is uniformly bounded in time. From here we may proceed similarly to above (using Lemma \ref{lem:HomogGrad} instead of Lemma \ref{lem:inhom_bds_CZ}) and eventually invoke Lemma \ref{dynamic}.  
If $\bar{q} > 1$ more thought is necessary since $\norm{u(t)}_{\bar{q}}$ is not a conserved quantity and the continuity argument used for convolution-type systems will need to be refined due to the presence of the lower-order term. 
From the previous argument in the $\bar{q} = 1$ case it is clear that the result will follow if for any $\epsilon > 0$, we may choose $\norm{u_0}_{\bar{q}}$ sufficiently small to ensure $\norm{u(t)}_{\bar{q}} < \epsilon$ uniformly.
To this end we note that \eqref{ineq:LpEvoSmallData} holds for all $\bar{q}$: 
\begin{align*}
\frac{d}{dt} \norm{u}_{\bar{q}}^{\bar{q}} \leq \left(C_1 - \delta\frac{C_2}{\norm{u}_{\bar{q}}^{\alpha_2(\bar{q}+1)}} \right)\norm{u}_{\bar{q}+1}^{\bar{q}+1} + C(\bar{q})M^{\bar{q}+1}. 
\end{align*}
The result will follow from a continuity argument. 
If $\norm{u_0}_{\bar{q}}$ is chosen such that 
\begin{equation*}
C_1 - \delta\frac{C_2}{\norm{u}_{\bar{q}}^{\alpha_2(\bar{q}+1)}} < \frac{\delta}{4}, 
\end{equation*}
then for some time (depending on $M$) this implies, 
\begin{align*}
\frac{d}{dt} \norm{u}_{\bar{q}}^{\bar{q}} & \leq -\frac{\delta}{2}\norm{u}_{\bar{q}+1}^{\bar{q}+1} + C(p)M^{\bar{q}+1} \\
& \leq -\frac{\delta}{2}\norm{u}_{\bar{q}}^{\bar{q}} + C(p)M^{\bar{q}+1} + \frac{\delta}{2}M, 
\end{align*} 
which in turn implies, 
\begin{equation*}
\norm{u(t)}_{\bar{q}}^{\bar{q}} \leq \max\left(\norm{u}_{\bar{q}}^{\bar{q}}, \frac{2}{\delta}\left( C(p)M^{\bar{q}+1} + \frac{\delta}{2}M\right)\right). 
\end{equation*}
From here it is clear that if the mass is chosen sufficiently small then $\norm{u(t)}_{\bar{q}}$ can be bounded uniformly by a sufficiently small constant. Hence we may proceed as after \eqref{ineq:LpEvoSmallData} above to finish the proof. 
\end{proof} 
%%%%%%%%%%%%%%%%%%%%%%%%%%%%%%%%%%%%%%%%%%%%%%%%%%%%%%%%%%%%%%%%%%%%%%%%%%%%%%%%%%%%%%%%%%%%%%%%%
\section{Large Data Theory for Critical Convolution-type Systems}
As in \cite{BRB10,BlanchetEJDE06,Blanchet09}, global existence in critical cases will follow from continuation (Theorem \ref{thm:Continuation}) in \cite{BRB10} provided $\int_{\set{u > k}} u \log u dx$ is bounded for large $k$, which will be deduced using the energy dissipation inequality \eqref{EnrDiss}. 
There are two difficulties here.  Firstly, due to \textbf{(D3)}, the nonlinear entropy 
$\int \Phi(u) dx$ and the Boltzmann entropy $\int u \log u dx$ are no longer uniformly bounded from below as the solution spreads out. 
Moreover, and perhaps more fundamentally, is the possibility of $\K$ not being bounded from below, which again can cause the potential energy to grow unboundedly as the solution spreads out.  
Both difficulties will be overcome by the following two standard lemmas. 
The first shows that the decay of the entropy is bounded from below by controllable quantities. 

\begin{lemma}[Entropy Lower Bound] \label{lem:entropy_bd}
Let $A$ be admissible, $u \in L_+^{1}(\Real^d;(1+\abs{x}^2)dx) \cap L^\infty(\Real^d)$ with $M = \norm{u}_1$. Then for all $\epsilon > 0$, 
\begin{equation} 
\int u \log u dx  \geq M\log\left(\frac{\epsilon^{d/2}M}{\pi^{d/2}}\right) - \epsilon\mathcal{M}_2(u), \label{ineq:Bentropy_below}
\end{equation}
and
\begin{equation}
\int \Phi(u) dx  \geq C_A M\log\left(\frac{\epsilon^{d/2}M}{\pi^{d/2}}\right) - \epsilon C_A\mathcal{M}_2(u). \label{ineq:NLentropy_below}
\end{equation}
\end{lemma} 
\begin{proof}
Following \cite{BlanchetEJDE06}, by Jensen's inequality for probability measures ($d\mu = \frac{u}{M}dx$), 
\begin{align*}
\int u(x) \log u(x) dx + \epsilon\int \abs{x}^2u(x) dx & = -M\int \log\left(\frac{e^{-\epsilon\abs{x}^2}}{u(x)}\right)\frac{u}{M} dx \\
& \geq -M\log \left(\frac{1}{M}\int e^{-\epsilon\abs{x}^2} dx\right) \\
& = M\log\left(\frac{\epsilon^{d/2}M}{\pi^{d/2}}\right)
\end{align*}
Therefore, \eqref{lem:entropy_bd} holds.  
Next, by \textbf{(D3)}, for some $\delta > 0$, $A^\prime(z) \leq C_A$ for $z < \delta$. 
Let $h(z) = \int_1^z A^\prime(s) s^{-1} ds$ and note that $\int \Phi(u) dx = \int_{\Real^d}\int_0^u h(z) dz dx$. 
For $z < 1$ we have,
\begin{align*}
h(z) & = -\int_z^1 \frac{A^\prime(s)}{s} ds \\
& \geq -\int_\delta^1\frac{A^\prime(s)}{s}ds - C_A\int_{\min(z,\delta)}^\delta\frac{1}{s}ds \\
& \geq - C + C_A\log z.
\end{align*}
Therefore, since $\log z$ is integrable at zero, we have the following by Chebyshev's inequality,
\begin{align*}
\int \Phi(u) dx & = \int \int_0^u h(z) dz dx\\
& \geq \int \mathbf{1}_{\set{u < 1}} \int_0^u h(z) dz + \mathbf{1}_{\set{u > 1}}\int_0^1 h(z) dz dx \\
& \geq \int \mathbf{1}_{\set{u < 1}}\left(C_A(u \log u - u) - Cu\right) - C\mathbf{1}_{\set{u > 1}} dx\\
& \geq C_A \int_{\set{u < 1}} u \log u dx - C(M).  
\end{align*}
Therefore, using \eqref{ineq:Bentropy_below} we obtain \eqref{ineq:NLentropy_below}.
\end{proof}

The following lemma establishes a uniform bound on the second moments for critical problems with $m^\star = 1$. 

\begin{lemma}[Second Moment Estimate] \label{lem:SME}
Let $A$ and $\K$ be admissible and critical with $m^\star = 1$. Then, 
\begin{equation*}
\mathcal{M}_2(t) \leq \mathcal{M}_2(0) + M\left(C_1 + C_2M\right)t
\end{equation*}
for some constants $C_i > 0$. 
\end{lemma} 
\begin{proof}
We argue formally, noting that the computations can easily be made rigorous with standard arguments. 
Computing the time evolution of the second moment, 
\begin{equation*}
\frac{d}{dt}\mathcal{M}_2(t) = 2d\int A(u) dx + \int\int u(x)u(y)(x-y)\cdot\grad \K(x-y) dx dy. 
\end{equation*}
By \textbf{(D1)} and \textbf{(D3)} of Definition \ref{def:admDiff}, and criticality in the sense of Definition \ref{def:criticality}, we necessarily have $A^\prime(z) \lesssim 1$, and hence $\int A(u) dx \lesssim M$. 
By Definition \ref{def:admK} and $m^\star = 1$ we have,
\begin{equation*}
\abs{(x-y)\cdot \grad \K(x-y)} \leq C.  
\end{equation*}
Therefore by integration, the lemma follows. 
\end{proof}

We now prove Theorem \ref{thm:GWP}. 
\begin{proof}(\textbf{Theorem \ref{thm:GWP}})
We proceed formally, noting that the arguments can be made rigorous with the regularization procedure of the local existence theory.
We begin by proving \textit{(i)}. We first prove the result for the case when either $\K$ is not bounded below or the diffusion is not degenerate. As noted above, we will not get uniform-in-time bounds. 
Recall the energy dissipation inequality \eqref{EnrDiss}, 
\begin{equation*}
\int \Phi(u) dx - \frac{1}{2}\int u \K \ast u dx \leq \F(u_0) := F_0. 
\end{equation*}
By the asymptotic expansion of the kernel assumed in Proposition \ref{prop:critical_mass} and \textbf{(BD)}, we have that for all $\epsilon > 0$, $\exists \;\delta,R >0$
such that, 
\begin{align*}
\int \Phi(u) dx + \frac{c + \epsilon}{2}\int\int_{\abs{x-y} < \delta}u(x)u(y)\log\abs{x-y}dxdy \\
\leq F_0 + \frac{1}{2}\int\int_{\delta < \abs{x-y} < R}u(x)u(y)\K(\abs{x-y})dxdy & - C\int\int_{\abs{x-y} > R}u(x)u(y)\log\abs{x-y}dxdy.
\end{align*}
Note that for $R > 0$ sufficiently large, 
\begin{align*}
\int\int_{\abs{x-y} > R}u(x)u(y)\abs{\log\abs{x-y}} dx dy & \leq  \frac{\log R}{R}\int\int_{\abs{x-y} > R}u(x)u(y)\abs{x-y} dx dy \\ 
& \lesssim \frac{\log R}{R} M^{3/2}\mathcal{M}_2(t)^{1/2},
\end{align*}
Therefore,
\begin{equation*}
\int \Phi(u) dx + (c+\epsilon)\frac{1}{2}\int\int_{\abs{x-y} < \delta}u(x)u(y)\log\abs{x-y}dxdy \leq F_0 + C(M,\delta,R) + C(R,M)\mathcal{M}_2(t)^{1/2}.
\end{equation*}
By the logarithmic Hardy-Littlewood-Sobolev inequality \eqref{ineq:log_sob_localized},
\begin{align*}
\int \Phi(u) dx - (c+\epsilon)\frac{M}{2d}\int u \log u dx \leq F_0 + C(M,\delta,R) + C(R,M)\mathcal{M}_2(t)^{1/2}.
\end{align*}
Therefore, choosing $k \geq 1$,  
\begin{equation*}
\int_{\set{u > k}} u\log u \left(\frac{\Phi(u)}{u\log u} - (c+\epsilon)\frac{M}{2d}\right) dx + \int_{u < k} \Phi(u) dx 
\leq F_0 + C(M,\delta,R) + C(R,M)\mathcal{M}_2(t)^{1/2}.  
\end{equation*}
By Lemma \ref{lem:entropy_bd},
\begin{equation*}
\int_{\set{u > k}} u\log u \left(\frac{\Phi(u)}{u\log u} - (c+\epsilon)\frac{M}{2d}\right) dx 
\leq F_0 + C(M,\delta,R) + C(R,M)\mathcal{M}_2(t). 
\end{equation*}
By Lemma \ref{lem:SME}, $\mathcal{M}_2(t) \lesssim 1 + t$ for all $t < \infty$. 
Since $M < M_c$ as defined in \eqref{def:MC}, it is possible to choose $k$ large enough and $\epsilon$ small enough such that $\int_{\set{u > k}}u \log u dx$
is bounded on any finite time interval. This is sufficient to imply equi-integrability on any finite time interval and hence by Theorem \ref{thm:Continuation} the solution $u(t)$ must be global.\\

We now refine the argument under the additional hypotheses that $\K$ is bounded below and the diffusion is degenerate. 
First note (see e.g. \cite{BRB10}) that $\int_0^1A^\prime(z)z^{-1}dz < \infty$ implies a uniform in time bound on the 
entropy: $\int \Phi(u)dx \gtrsim -C(M)$.
Choosing $k \geq 1$ this implies with the energy dissipation inequality \eqref{EnrDiss},
\begin{align*}
\int_{u > k} \Phi(u) dx - \frac{1}{2}\int u \K \ast u dx \leq F_0 + C(M). 
\end{align*}
Next using the asymptotic expansion of $\K$ at the origin we may choose $\epsilon$ and $\delta$ such that 
\begin{align*}
\int_{u > k} \Phi(u) dx + \frac{c+\epsilon}{2}\int\int_{ \set{u > k} \cap \set{\abs{x-y} < \delta}} u(x)u(y) \log\abs{x-y} dx dy \\ \leq F_0 + C(M) + \sup_{\abs{x}>\delta}\abs{\K(\abs{x})}M^2 + \left(\int_{\abs{x} < \delta} \abs{\K(x)} dx\right)kM. 
\end{align*}
Finally applying the logarithmic HLS \eqref{ineq:log_sob_localized} we have, 
\begin{align*}
\int_{\set{u > k}} u\log u \left(\frac{\Phi(u)}{u\log u} - (c+\epsilon)\frac{M}{2d}\right) dx \leq F_0 + C(M) + \sup_{\abs{x}>\delta}\abs{\K(\abs{x})}M^2 + \left(\int_{\abs{x} < \delta} \abs{\K(x)} dx\right)kM. 
\end{align*}
Therefore by choosing $k$ sufficiently large and $\epsilon$ sufficiently small (since $M < M_c$) we may bound $\int_{u > k}u \log u dx$ uniformly in time. By Theorem \ref{thm:Continuation} the solution is global and uniformly bounded.  \\

We now turn to the proof of \textit{(ii)}, which shows that with additional homogeneity assumptions the energy dissipation inequality can be used to deduce optimal decay results for critical problems when $M < M_c$. Recall this result is already known \cite{BlanchetEJDE06,Blanchet09}. 
The key is that the scaling invariance can be used to transform the energy dissipation inequality into something significantly stronger.
We proceed by considering the self-similar variables, as in \cite{BlanchetEJDE06,Blanchet09,BlanchetDEF10,BedrossianIA10},
defining $\theta(\tau,\eta)$ such that 
\begin{equation}
e^{-d\tau}\theta(\tau,\eta) = u(t,x), \label{def:theta}
\end{equation}
with coordinates $e^\tau\eta = x$ and $e^{d\tau} - 1 = dt$. 
In these coordinates, if $u(t,x)$ solves \eqref{def:ADD}, by the homogeneity of the Newtonian potential we have, 
\begin{equation}
\partial_\tau \theta = \grad \cdot(\eta \theta) + \Delta \theta^{2-2/d} - \grad \cdot (\theta \grad \mathcal{N} \ast \theta). \label{eq:asymptotic_pde} 
\end{equation}

Note by definition, uniform boundedness of $\theta(\tau)$ is equivalent the decay stated in \textit{(ii)} of Theorem \ref{thm:GWP}. 
The key here is that the assumed homogeneity implies that an analogue of the energy dissipation inequality \eqref{EnrDiss} still holds (see \cite{Blanchet09} for more details). In $d \geq 3$,  
\begin{align*}
\frac{1}{1-2/d}\int \theta^{2-2/d}d\eta + \frac{1}{2}\int \theta(\tau,\eta)\abs{\eta}^2 d\eta - \frac{1}{2}\int\int \theta(\tau,\eta) \theta(\tau,\zeta) \mathcal{N}(\eta - \zeta) d\eta d\zeta \\ \leq \frac{1}{1-2/d}\int \theta_0^{2-2/d}d\eta + \frac{1}{2}\int \theta_0(\eta)\abs{\eta}^2 d\eta - \frac{1}{2}\int\int \theta_0(\eta) \theta_0(\zeta) \mathcal{N}(\eta - \zeta) d\eta d\zeta := \mathcal{G}_0 < \infty. 
\end{align*}

If $d \geq 3$ then the argument using the Hardy-Littlewood-Sobolev inequality in \cite{Blanchet09} applies and easily proves the uniform equi-integrability of $\theta(\tau,\eta)$. This, in turn, can be shown to imply the uniform $L^\infty$ bound by an easy variant of Theorem \ref{thm:Continuation} (see e.g. \cite{BedrossianIA10}). 
In $\Real^2$, the modified free energy becomes, 
\begin{equation*}
\mathcal{G}(\theta) := \int \theta \log \theta d\eta + \frac{1}{2}\int \abs{\eta}^2\theta d\theta - \frac{1}{4\pi}\int\int \theta(\eta)\theta(\zeta) \log\abs{\eta-\zeta} d\eta d\zeta, 
\end{equation*}
and the modified energy dissipation inequaliy $\mathcal{G}(\theta(\tau)) \leq \mathcal{G}(\theta_0) := \mathcal{G}_0$ holds. 
The key is to use the additional second moment term in the modified free energy and Lemma \ref{lem:entropy_bd} to control the possiblity of the entropy or potential energies from being unbounded from below, a trick which was not evidently possible in the arguments of Part \textit{(i)}. 
Using the logarithmic Hardy-Littlewood-Sobolev inequality \eqref{ineq:log_sob}, 
\begin{align*}
\left(1 - \frac{M}{8\pi}\right) \int \theta(\tau,\eta)\log \theta(\tau,\eta) d\eta & +  \frac{1}{2}\mathcal{M}_2(\theta(\tau)) \leq \mathcal{G}_0 + C(M).  
\end{align*} 

Assume further of course that $M < M_c = 8\pi$. 
Then by Lemma \ref{lem:entropy_bd}, for all $\epsilon > 0$,  
\begin{align*}
-\left(1 - \frac{M}{8\pi}\right)\epsilon\mathcal{M}_2(\theta(\tau)) + \left(1 - \frac{M}{8\pi}\right) \int \left(\theta(\tau,\eta)\log \theta(\tau,\eta)\right)_+ d\eta & +  \frac{1}{2}\mathcal{M}_2(\theta(\tau))  \\ \leq \mathcal{G}_0 - \left(1 - \frac{M}{8\pi}\right)M\log\left(\frac{\epsilon^{d/2}M}{\pi^{d/2}}\right) + C(M). 
\end{align*} 
Hence, if we choose $\epsilon < 2^{-1}\left(1 - \frac{M}{8\pi}\right)^{-1}$ we have, 
\begin{equation*}
\left(1 - \frac{M}{8\pi}\right) \int \left(\theta(\tau,\eta)\log \theta(\tau,\eta)\right)_+ d\eta \leq \mathcal{G}_0 - \left(1 - \frac{M}{8\pi}\right)M\log\left(\frac{\epsilon^{d/2}M}{\pi^{d/2}}\right) + C(M),
\end{equation*}
for all $\tau$. 
Hence $\theta(\tau)$ is uniformly equi-integrable and the result again follows from variants of standard continuation arguments. 
\end{proof}

%%%%%%%%%%%%%%%%%%%%%%%%%%%%%%%%%%%%%%%%%%%%%%%%%%%%%%%%%%%%%%%%%%%%%%%%%%%%%%%%%%%%%%%%%%%%%%%%%
\section{Appendix: Elliptic Estimates}
In this section we discuss basic estimates regarding the elliptic system 
\begin{equation}
-\grad \cdot (a(x) \grad c) + \gamma(x) c = f, \label{eq:cPDE}
\end{equation}
defined on all of $\Real^d$. 
We assume such estimates, or better ones, can be found elsewhere in the classical literature or in textbooks but we could not locate them. 
The ensuing proofs are technically bootstrap arguments, since they require certain norms to be a priori finite to begin with. 
However, the strong solution which vanishes at infinity that we are interested in is constructed by solving \eqref{eq:cPDE} on successively larger balls of radius $R$ with zero Dirichlet conditions. 
On each ball the estimates of Lemmas \ref{lem:EllipticLp} \ref{lem:HomogGrad} and \ref{lem:inhom_bds_CZ} can be made independent of $R$ and hence can be passed to the limit. We skip these straightforward details and only give the formal estimates.
\begin{lemma} \label{lem:EllipticLp}
Let $c$ be the strong solution that vanishes at infinity in $\Real^d$ of
\begin{equation*}
-\grad \cdot (a(x)\grad c) + \gamma(x)c = f,  
\end{equation*} 
for $a(x) \in L^\infty$ strictly positive and $\gamma(x) \in L^\infty$ non-negative. 
If $d \geq 3$, then for $\frac{2}{d} + \frac{1}{p} = \frac{1}{q}$, $1 < q < \frac{d}{2}$, $\frac{d}{d-2} < p < \infty$ we have 
\begin{equation}
\norm{c}_p \leq \frac{C(p,d)}{\inf a}\norm{f}_q. \label{ineq:HomogCEstimate}
\end{equation}
Moreover, as $p \rightarrow \infty$, $C(p,d) \lesssim p$.  
If $d \geq 2$, we have the $\dot{H}^{-1}$ stability estimate: if $f = \grad \cdot F$ then,
\begin{equation*}
\norm{\grad c}_2 \leq \frac{1}{\inf a}\norm{F}_{2}. 
\end{equation*}
If $d\geq 2$ and $\gamma$ is also strictly positive then for all $1 < p < \infty$, we have the estimate 
\begin{equation}
\norm{c}_p \leq \frac{1}{\inf \gamma}\norm{f}_p. \label{ineq:InhomogC}
\end{equation}
\end{lemma}
\begin{proof}
Note this estimate follows from the $L^{p,\infty}$ Young's inequality and the representation of $c$ as a convolution in the cases when both $a$ and $\gamma$ are constants. 
Define $\alpha = (d-2)p/d - 1$ then multiply by $c^{\alpha}$ and integrate:  
\begin{align*}
-\int c^{\alpha}\grad \cdot a(x) \grad c dx + \int \gamma(x)c^{\alpha+1} dx & = \int c^{\alpha}f dx.  
\end{align*}
Use inverse chain rule and Sobolev embedding on the LHS (dropping the low order term):
\begin{align*}
\alpha\int a(x)c^{\alpha-1}\abs{\grad c}^2 dx & \leq \int c^{\alpha}f dx \\
\frac{4\alpha}{(\alpha + 1)^2}\int \abs{\grad (c^{\frac{\alpha + 1}{2}})}^2 dx & \leq \frac{1}{\inf a}\int c^{\alpha}f dx \\
\norm{c^{\frac{\alpha+1}{2}}}_{\frac{2d}{d-2}}^2 & \lesssim \frac{1}{\inf a}\int c^{\alpha}f dx \\
\norm{c}^{\alpha + 1}_{\frac{d}{d-2}(\alpha + 1)} & \lesssim \frac{1}{\inf a}\int c^{\alpha}f dx \\
\norm{c}^{\alpha + 1}_{\frac{d}{d-2}(\alpha + 1)} & \lesssim \frac{1}{\inf a}\norm{c}_{\frac{d}{d-2}(\alpha+1)}^\alpha\norm{f}_{\frac{dp}{d+2p}}. 
\end{align*}
By definition of $\alpha$, this is in fact 
\begin{equation*}
\norm{c}_p \lesssim \frac{1}{\inf a}\norm{f}_q. 
\end{equation*}
The second estimate is trivial: 
\begin{align*}
\int a(x) \abs{\grad c}^2 + \gamma(x) c^2 & = -\int F\cdot \grad c dx \\ 
(\inf a)\norm{\grad c}_2^2 & \leq \norm{F}_2 \norm{\grad c}. 
\end{align*}
The third estimate is similarly trivial: 
\begin{align*}
(p-1)\int c^{p-2}\abs{\grad c} a(x) dx + \int \gamma(x) c^{p} dx & = \int c^{p-1} f dx \\
(p-1)\left(\frac{4}{p^2}\right)\int a(x)\abs{\grad c^{p/2}}^2 dx + (\inf \gamma)\norm{c}_p^p & \leq \norm{c}_p^{p-1}\norm{f}_p \\
\norm{c}_p & \leq \frac{1}{\inf \gamma}\norm{f}_p. 
\end{align*}
\end{proof}

\begin{lemma}[Homogeneous gradient estimates] \label{lem:HomogGrad}
Let $d \geq 3$ and let $c$ be a strong solution which vanishes at infinity in $\Real^d$ of 
\begin{equation*}
-\grad \cdot (a(x)\grad c) = f.    
\end{equation*} 
Suppose $a(x) \in C^1$ is strictly positive, bounded and $\grad a$ is also uniformly bounded.
Then if $p > d/(d-2)$ we have,  
\begin{equation*}
\norm{D^2c}_p \leq C(a,p,d)\left(\norm{f}_p + \norm{f}_{\frac{pd}{2p+d}}\right). 
\end{equation*}
Moreover, as $p \rightarrow \infty$, $C(a,p,d) \lesssim p$. 
\end{lemma}
\begin{remark}
It is crucial to note that $1 < pd/(2p+d) \leq p$. 
\end{remark}
\begin{proof}
 The proof is a variant of Theorem 9.11 in \cite{GT} adapted to $\Real^d$. 
The following Gagliardo-Nirenberg-type inequality, (see e.g. Theorem 7.28 in \cite{GT}), will be used:  
\begin{equation}
\norm{\grad g}_p \lesssim_{d,p} \norm{g}_p^{1/2}\norm{D^2g}_p^{1/2} \lesssim \epsilon\norm{D^2g}_p + \frac{1}{\epsilon}\norm{g}_p.  \label{ineq:WeightSE2}
\end{equation}
Consider a ball of radius $R > 0$ centered at a point $x_0 \in \Real^d$, denoted $B_R(x_0)$. 
By the Calder\'on-Zygmund inequality we have for a sufficiently smooth $v$ supported in $B_R(x_0)$ and $1 < p < \infty$,  
\begin{equation*}
\norm{D^2v}_{L^p(B_R(x_0))} \lesssim a(x_0)\norm{\Delta v}_p. 
\end{equation*}
Hence, 
\begin{align*}
\norm{D^2v}_p \lesssim \norm{\grad \cdot ((a(x_0) - a(x))\grad v)}_p + \norm{\grad \cdot (a(x) \grad v)}_p. 
\end{align*}
We control the first term using \ref{ineq:WeightSE2}, for some $\epsilon > 0$, 
\begin{align*}
\norm{\grad \cdot (a(x_0) - a(x))\grad v}_p &\leq \sup_{x \in B_R(x_0)}\abs{a(x_0) - a(x)}\norm{\Delta v}_p + \norm{\grad a \cdot \grad v}_p \\
&\leq \sup_{x \in B_R(x_0)}\abs{a(x_0) - a(x)}\norm{D^2 v}_p + \norm{\grad a}_{\infty} \norm{\grad v}_{p} \\ 
&\leq \sup_{x\in B_R(x_0)}\abs{a(x_0) - a(x)}\norm{D^2 v}_p + \norm{\grad a}_{\infty}\left(\epsilon\norm{D^2v}_{p} + \frac{C}{\epsilon}\norm{v}_p\right).  
\end{align*}
Hence, choosing $R$ sufficiently small, depending on the smoothness of $a(x)$ and choosing $\epsilon$ sufficiently small, 
we can then deduce,
\begin{equation*}
\norm{D^2v}_p \lesssim \norm{\grad \cdot (a(x)\grad v)}_p + \norm{v}_p. 
\end{equation*} 
Choosing $v = \eta c$ in the above inequality gives, for any $\sigma \in (0,1)$,
\begin{align*}
\norm{D^2 c}_{L^p(B_{\sigma R}(x_0))} & \lesssim \norm{\grad \cdot a(x)\grad c}_{L^p(B_{\sigma R}(x_0))} + \norm{a\grad \eta \cdot \grad c \\ &\quad+ ca\Delta\eta + c\grad a \cdot \grad \eta}_p + \norm{\grad a}_{L^\infty(B_{\sigma R})} \norm{c}_{L^p(B_{\sigma R})} \\ 
& \quad \leq \norm{f}_{L^p(B_{\sigma R}(x_0))} + \frac{C}{R}\norm{a \grad c}_{L^p(B_{\sigma R}(x_0))} + \frac{C}{R^2}\norm{ca}_{L^p(B_{\sigma R}(x_0))} + \frac{C}{R}\norm{c\grad a}_{L^p(B_{\sigma R}(x_0))} \\ & + \norm{\grad a}_{L^\infty(B_{\sigma R}(x_0))} \norm{c}_{L^p(B_{\sigma R}(x_0))}.
\end{align*}
We now have to deal with the latter error terms. 
Using the interpolation inequality on page 250 of \cite{GT}, which is essentially just \eqref{ineq:WeightSE2}, we have, 
\begin{equation*}
\norm{a \grad c}_{L^p(B_{\sigma R}(x_0))} \leq \norm{a}_\infty\norm{\grad c}_{L^p(B_{\sigma R}(x_0))} \leq \norm{a}_\infty R \epsilon\norm{D^2c}_{L^p(B_{\sigma R}(x_0))} + \norm{a}_\infty\frac{C}{R\epsilon}\norm{c}_{L^p(B_{\sigma R}(x_0))},
\end{equation*}
where crucially the constant $C$ does not depend on $\sigma$ or $R$ ($C$ is some power of a constant of a Gagliardo-Nirenberg inequality which depends only on the geometry and not the diameter). 
Hence by choosing $\epsilon$ sufficiently small independent of $R$ we have,  
\begin{equation}
\norm{D^2 c}_{L^p(B_{\sigma R}(x_0))} \lesssim \norm{f}_{L^p(B_{\sigma R}(x_0))} + \left(\frac{1}{R^2} + \frac{1}{R} + \norm{\grad a}_{L^\infty(B_{\sigma R}(x_0))}\right)\norm{c}_{L^p(B_{\sigma R}(x_0))}. \label{ineq:key2order}
\end{equation}
Since $p > d/(d-2)$, \eqref{ineq:HomogCEstimate} ensures that latter term will add up as we cover $\Real^d$ with balls (using that we may take $R \gtrsim 1$ due to the bound on $\grad a$). 
Therefore, 
\begin{align*}
\norm{D^2 c}_{p} \lesssim \norm{f}_{p} + \norm{f}_{\frac{dp}{2p+d}}. 
\end{align*}
\end{proof} 

\begin{lemma} \label{lem:inhom_bds_CZ}
Let $d \geq 2$ and in $\Real^d$ let $c$ be a strong solution which vanishes at infinity in $\Real^d$ of
\begin{equation*}
-\grad \cdot (a(x)\grad c) + \gamma(x)c = f,  
\end{equation*} 
for $\gamma(x) \in L^\infty$ strictly positive and $a(x) \in C^1$ strictly positve, bounded and $\grad a$ is uniformly bounded. 
Then we have the gradient estimate, 
\begin{equation*}
\norm{D^2c}_{p} \leq C(a,p,d)\norm{f}_p,   
\end{equation*}
with $C(a,p,d) \lesssim p$ as $p \rightarrow \infty$. 
\end{lemma} 
\begin{proof}
Easy variant of the above using the better $L^p$ norm estimate \eqref{ineq:InhomogC} to control the lower order terms in \eqref{ineq:key2order}.   
\end{proof}

%%%%%%%%%%%%%%%%%%%%%%%%%%%%%%%%%%%%%%%%%%%%%%%%%%%%%%%%%%%%%%%%%%%%%%%%%%%%%%%%%%%%%%%%%%%%%%%%%
\section{Appendix: Gagliardo-Nirenberg}
\begin{lemma}[Homogeneous Gagliardo-Nirenberg] \label{lem:GNS}
Let $d \geq 2$ and $f:\Real^d \rightarrow \Real$ satisfy $f \in L^p\cap L^q$ and $\grad f^k \in L^r$. Moreover let $1 \leq p \leq rk \leq dk$, $k < q < rkd/(d-r)$ and
\begin{equation}
\frac{1}{r} - \frac{k}{q} - \frac{s}{d} < 0. \label{cond:GNS}
\end{equation}
Then there exists a constant $C_{GNS}$ which depends on $s,p,q,r,d$ such that
\begin{equation}
\norm{f}_{L^q} \leq C_{GNS}\norm{f}^{\alpha_2}_{L^p} \norm{f^k}^{\alpha_1}_{\dot{W}^{s,r}}, \label{eq:GNS}
\end{equation}
where $0 < \alpha_i$ satisfy
\begin{equation}
1 = \alpha_1 k + \alpha_2,
\end{equation}
and
\begin{equation}
\frac{1}{q} - \frac{1}{p} = \alpha_1(\frac{-s}{d} + \frac{1}{r} - \frac{k}{p}).
\end{equation}
\end{lemma}

\bibliographystyle{plain}
\bibliography{nonlocal_eqns}

\end{document}